\documentclass[a4paper,12pt]{amsart}

\usepackage[
  margin=30mm,
  marginparwidth=25mm,     
  marginparsep=2mm,       
  bottom=25mm,
  ]{geometry}

\usepackage[latin1]{inputenc}
\usepackage[T1]{fontenc}
\usepackage[english]{babel}
\usepackage{amsmath,amssymb,amsthm,mathtools}

\expandafter\let\expandafter\xbf\csname bfseries \endcsname
\expandafter\let\expandafter\xmd\csname mdseries \endcsname
\let\xbar\bar
\let\bar\xbar
\expandafter\let\csname bfseries \endcsname\xbf
\expandafter\let\csname mdseries \endcsname\xmd

\usepackage{latexsym}
\usepackage{delarray}
\usepackage{bbm}
\usepackage{scrextend}
\usepackage{hyperref}
\usepackage[pdftex,usenames,dvipsnames]{xcolor}
\usepackage{datetime}
\usepackage{mathrsfs}
\usepackage{enumitem}
\usepackage{tikz}

\usepackage{MnSymbol,wasysym}

\usepackage{multicol}

\newtheorem{Th}{Theorem}[section]
\newtheorem{Prop}[Th]{Proposition}
\newtheorem{Lem}[Th]{Lemma}
\newtheorem{Cor}[Th]{Corollary}
\newtheorem{Rem}[Th]{Remark}

\newtheorem*{Rem*}{Remark}

\newcommand{\Rn}{\mathbb{R}^n}
\newcommand{\NN}{\mathbb{N}}
\newcommand{\WW}{\mathbb{W}}
\newcommand{\CC}{\mathbb{C}}
\newcommand{\ZZ}{\mathbb{Z}}

\DeclareMathOperator{\supp}{supp}

\allowdisplaybreaks 

\title
[Fourier-Bessel Variation operators]
{Variation operators for semigroups associated with Fourier-Bessel expansions}

\author[J. J. Betancor]{J. J. Betancor}
\address{Jorge J. Betancor\newline
	Departamento de An\'alisis Matem\'atico, Universidad de La Laguna,\newline
	Campus de Anchieta, Avda. Astrof\'isico S\'anchez, s/n,\newline
	38721 La Laguna (Sta. Cruz de Tenerife), Spain}
\email{jbetanco@ull.es}

\author[A. J. Castro]{A. J. Castro}
\address{\newline
       Alejandro J. Castro \newline
       Department of  Mathematics,
       Nazarbayev University, \newline
	   Kabanbay Batyr Ave. 53, Nur-Sultan 010000 Kazakhstan}
\email{alejandro.castilla@nu.edu.kz}

\author[M. De Le\'on-Contreras]{M. De Le\'on-Contreras}
\address{\newline
       Marta De Le\'on-Contreras \newline
       Department of  Mathematics and Statistics,
       University of Reading, \newline
	   Reading RG6 6AX, United Kingdom}
\email{m.deleoncontreras@reading.ac.uk }


\subjclass[2010]
{}

 \thanks{
J. J. B. was partially supported by PID2019-106093GB-I00,
A. J. C.  by the Nazarbayev University FDCRGP  110119FD4544 and 
M. D L-C by EPSRC Research Grant EP/S029486/1.}

\begin{document}

\maketitle

\begin{abstract}
In this paper we establish $L^p$-boundedness properties for variation operators defined by semigroups associated with Fourier-Bessel expansions.
\end{abstract}

\setcounter{secnumdepth}{3}
\setcounter{tocdepth}{3}


\section{Introduction}

Variation norms allow us to estimate, in certain sense, the fluctuations of a given family of operators. The results obtained by using variation norms are better than the ones deduced by square functions. Variation inequalities can be used to measure the speed of convergence of the underlying family of operators. Also, they allow us to obtain pointwise convergence without using the Banach principle that is needed when maximal inequalities are considered.\\

L\'epingle (\cite{Le}) proved the first variational inequality for martingales.
Later, Bourgain (\cite{Bou}) established them in the ergodic theory context for $L^2$. This work was extended to $L^p$, $1 \leq p < \infty$, by 
Jones, Kaufman, Rosenblatt and Wierdl (\cite{JKRW}).
Bourgain's work opened new lines of research. One of such was the study of variation inequalities in harmonic analysis.
Campbell, Jones, Reinhold and Wierdl
(\cite{CJRW1}) proved these inequalities for the Hilbert transform.
In \cite{CJRW2} they extended the results to singular integrals in higher dimensions. 
Since then, variation inequalities have been established for semigroups of operators and singular integrals in several settings
(see for example 
\cite{
BFHR,
CMMTV,
DMT,
GT,
HM,
HLP,
JSW,
JW,
LeMX,
MTX,
MT,
OSTTW,
PX}).\\

Suppose that $\{T_t\}_{t>0}$ is a family of operators defined in $L^p(\Omega)$, for some 
$1 \leq p \leq \infty$ and 
$\Omega \subseteq \Rn$.
If $\{t_j\}_{j \in \NN}$ is a decreasing 
sequence in $(0,\infty)$, the \textit{oscillation operator} 
$\mathcal{O}(\{T_t\}_{t>0},\{t_j\}_{j \in \NN})$ is defined by
\begin{equation*}
\mathcal{O}(\{T_t\}_{t>0},\{t_j\}_{j \in \NN})(f)(x)
:=
\Big(
\sum_{j=1}^\infty
\sup_{t_{j+1} \leq \varepsilon_{j+1} < \varepsilon_j \leq t_j}
\Big| 
T_{\varepsilon_j}(f)(x)
-
T_{\varepsilon_{j+1}}(f)(x)
\Big|^2
\Big)^{1/2}, \quad x \in \Omega.
\end{equation*}

Let $\rho>2$. The 
\textit{variation operator}
$\mathcal{V}_\rho(\{T_t\}_{t>0})$ is defined by
\begin{equation*}
\mathcal{V}_\rho(\{T_t\}_{t>0})(f)(x)
:=
\sup_{\{\varepsilon_j\}_{j \in \NN} \searrow}
\Big(
\sum_{j=1}^\infty
\Big| 
T_{\varepsilon_j}(f)(x)
-
T_{\varepsilon_{j+1}}(f)(x)
\Big|^\rho
\Big)^{1/\rho}, \quad x \in \Omega,
\end{equation*}
where the supremum is taken over all decreasing sequences $\{\varepsilon_j\}_{j \in \NN}$ in 
$(0,\infty)$.\\

It is usual to assume that $\rho>2$ in order to get $L^p$-boundedness for $\mathcal{V}_\rho$
(see \cite{Q}). We can define oscillation operators where the exponent $2$ is replaced by $r>2$. Then, the new $r$-operators are dominated by the $2$-operator. The $L^p$-boundedness of the $r$-oscillation operator when $r<2$ usually fails (see \cite{AJS}).\\

For every $k \in \NN$, consider
\begin{equation*}
\mathcal{V}_k(\{T_t\}_{t>0})(f)(x)
:=
\sup_{
\substack{2^{-k} < \varepsilon_\ell < \varepsilon_{\ell-1} < \dots < \varepsilon_1 <2^{-k+1} \\
\ell \in \NN}}
\Big(
\sum_{j=1}^{\ell-1}
\Big| 
T_{\varepsilon_j}(f)(x)
-
T_{\varepsilon_{j+1}}(f)(x)
\Big|^2
\Big)^{1/2}, \quad x \in \Omega.
\end{equation*}
The \textit{short variation operator} 
$\mathcal{S}_V(\{T_t\}_{t>0})$ is defined then by
\begin{equation*}
\mathcal{S}_V(\{T_t\}_{t>0})(f)(x)
:= 
\Big( 
\sum_{k=-\infty}^\infty
\Big|
\mathcal{V}_k(\{T_t\}_{t>0})(f)(x)
\Big|^2
\Big)^{1/2}, \quad x \in \Omega.
\end{equation*}

As in \cite[p. 60]{CJRW1}, it is convenient to comment that the operators $\mathcal{O}$,
$\mathcal{V}$ and $\mathcal{S}_V$ define measurable functions provided that $\{T_t\}_{t>0}$ satisfies some continuity properties  with respect to the parameter $t$.\\

For every $\lambda>0$, the \textit{$\lambda$-jump operator} is given by
\begin{align*}
\Lambda(\{T_t\}_{t>0},\lambda)(f)(x)
& :=
\sup\Big\{
n \in \NN \text{ : there exist }
s_1 < t_1 \leq s_2 < t_2 < \dots \leq s_n < t_n \\
& \hspace{1.5cm}
\text{such that } 
\Big| T_{t_i}(f)(x) - T_{s_i}(f)(x)\Big|> \lambda, \quad i=1, 2, \dots, n
\Big\}.
\end{align*}
If for some $x \in \Omega$ there exists the limit
$$\lim_{t \to 0^+} T_t(f)(x)$$
then, for every $\lambda>0$,
$$\Lambda(\{T_t\}_{t>0},\lambda)(f)(x)<\infty,$$
and 
$\Lambda(\{T_t\}_{t>0},\lambda)(f)(x)$
gives information about the convergence of $\{T_t(f)(x)\}_{t>0}$.\\

In this paper we study variation operators defined by semigroups of operators associated with Fourier-Bessel expansions.\\

Let $\nu>-1$. We consider the Bessel operator
$$B_\nu
:=
-\frac{d^2}{dx^2} - \frac{2 \nu + 1}{x} \frac{d}{dx}
\quad \text{on} \quad (0,1).$$
Let $J_\nu$ represents the Bessel function of the first kind and order $\nu$ and denote by 
$\{\lambda_{n,\nu}\}_{n \in \mathbb{N}}$ the sequence of positive zeros of $J_\nu$, with
$\lambda_{n+1,\nu} > \lambda_{n,\nu}$, 
$n \in \mathbb{N}$. We define, for every $n \in \mathbb{N}$,
$$\phi_n^\nu(x)
:=
d_{n,\nu} \,
\lambda_{n,\nu}^{1/2} \,
J_\nu(\lambda_{n,\nu} x) \,
x^{-\nu}, \quad x \in (0,1),$$
where
$$d_{n,\nu}
:= \frac{\sqrt{2}}{|\lambda_{n,\nu}^{1/2} J_{\nu+1}(\lambda_{n,\nu})|}.$$
We have that 
$$
B_\nu \phi_n^\nu
=
\lambda_{n,\nu}^2 \phi_n^\nu, \quad n \in \mathbb{N}.$$
The sequence $\{\phi_n^\nu\}_{n \in \mathbb{N}}$ is a complete orthonormal basis in 
$L^2((0,1),x^{2 \nu + 1}dx)$ (\cite{Ho}).
For every $n \in \mathbb{N}$, consider also
$$
c_n^\nu(f)
:=
\int_0^1 
\phi_n^\nu(x)
f(x)
x^{2 \nu + 1} \, dx
, \quad f \in L^2((0,1),x^{2 \nu + 1}dx).$$

We define
$$\Delta_\nu f
:=
\sum_{n=1}^\infty 
\lambda_{n,\nu}^2 
c_n^\nu(f)
\phi_n^\nu, 
\quad f \in D(\Delta_\nu),$$
where
$$D(\Delta_\nu)
:=
\Big\{ f \in L^2((0,1),x^{2 \nu + 1}dx)
\text{ : }
\sum_{n=1}^\infty 
|\lambda_{n,\nu}^2 
c_n^\nu(f)|^2 < \infty
\Big\}$$
is the domain of $\Delta_\nu$.
It is clear that 
$$\Delta_\nu f
=
B_\nu f, \quad
f \in C^\infty_c(0,1),$$
where $C^\infty_c(0,1)$ denotes the space of smooth functions with compact support in $(0,1)$. $\Delta_\nu$ is symmetric and positive in $L^2((0,1),x^{2 \nu + 1}dx)$. The operator
$-\Delta_\nu$ generates the semigroup of operators
$\{W_t^\nu\}_{t>0}$, where 
$$
W_t^\nu(f)
:=
\sum_{n=1}^\infty
e^{-t \lambda_{n,\nu}^2} c_n^\nu(f) \phi_n^\nu,
\quad t>0 
\quad \text{and} \quad
f \in L^2((0,1),x^{2 \nu + 1}dx).
$$
For every $t>0$ and $f \in L^2((0,1),x^{2 \nu + 1}dx)$, we have that
\begin{equation}\label{eq:1.1}
W_t^\nu(f)(x)
=
\int_0^1 W_t^\nu(x,y) f(y) y^{2\nu+1} \, dy, \quad x\in (0,1),
\end{equation}
where
$$
W_t^\nu(x,y)
:=
\sum_{n=1}^\infty
e^{-t \lambda_{n,\nu}^2} \phi_n^\nu(x) \phi_n^\nu(y),\quad x,y\in (0,1).
$$

According to \cite[Theorem 1.1]{MSZ} (see also 
\cite[Theorem 3.3]{NR1}), for every $t>0$, the operator
$W_t^\nu$ given by \eqref{eq:1.1} is bounded from
$L^p((0,1),x^{2 \nu + 1}dx)$ into itself, for every $1 \leq p \leq \infty$. The maximal operator $W_*^\nu$ defined by
$$
W_*^\nu f
:=
\sup_{t>0} |W_t^\nu(f)|$$
is bounded from $L^p((0,1),x^{2 \nu + 1}dx)$ into itself, for
every $1<p<\infty$, and from 
$L^1((0,1),x^{2 \nu + 1}dx)$ into 
$L^{1,\infty}((0,1),x^{2 \nu + 1}dx)$ (\cite[Theorem 4.1]{NR1}).\\

The Poisson semigroup $\{P_t^\nu\}_{t>0}$ generated by
$-\sqrt{\Delta_\nu}$ is given by
$$
P_t^\nu(f)
:=
\sum_{n=1}^\infty
e^{-t \lambda_{n,\nu}} c_n^\nu(f) \phi_n^\nu, \quad t>0.
$$
Also, by using the subordination formula we can write, for $ t>0$,
\begin{equation}\label{eq:subordination}
P_t^\nu(f)
=
\frac{t}{2\sqrt{\pi}}
\int_0^\infty 
\frac{e^{-t^2/(4u)}}{u^{3/2}}
W_u^\nu(f)
\, du=\int_0^1 P_t^\nu(x,y)f(y)y^{2\nu+1}\;dy, 
\end{equation}
where $$
P_t^\nu(x,y)
:=\int_0^\infty 
\frac{e^{-t^2/(4u)}}{u^{3/2}}
W_u^\nu(x,y)
\, du, \: x,y\in (0,1).
$$
The maximal operator
$$
P_*^\nu f
:=
\sup_{t>0} |P_t^\nu(f)|$$
is bounded from $L^p((0,1),x^{2 \nu + 1}dx)$ into itself, for
every $1<p<\infty$, and from $L^1((0,1),x^{2 \nu + 1}dx)$ into 
$L^{1,\infty}((0,1),x^{2 \nu + 1}dx)$.\\

The harmonic analysis associated with $\{\phi_n^\nu\}_{n \in \mathbb{N}} $-expansions started in the celebrated paper by Muckenhoupt and Stein \cite{MS}.
In the last decade harmonic analysis operators 
(Riesz transforms, Littlewood-Paley functions, multipliers and transplantation operators) in the 
$\{\phi_n^\nu\}_{n \in \mathbb{N}} $-setting have been studied in 
\cite{ABDR,BS1,BS2,CR3,CR2,CR1,CS1,CS2,NR3}.

Let $\beta\ge 0$ and $m \in \mathbb{N}$ such that $m-1 \leq \beta < m$. Suppose that $h \in C^m(0,\infty)$. The $\beta$-th Weyl derivative  $\mathbb{D}^\beta h$ is defined by
\begin{equation*}
\mathbb{D}^\beta h(t)
:=
\frac{-1}{\Gamma(m-\beta)}
\int_0^\infty h^{(m)}(t+s) s^{m-\beta-1} \, ds, 
\quad t >0,
\end{equation*}
whenever $\displaystyle\int_0^\infty |h^{(m)}(t+s)| s^{m-\beta-1} \, ds<\infty, \quad t >0.$

In particular, let $\beta\in\NN_0:=\NN\cup\{0\}$. We take $m=\beta+1$. Suppose that $t>0$ and  $\displaystyle\int_0^\infty |h^{(m)}(t+s)| s^{m-\beta-1} \, ds<\infty.$ Then, there exists the limit
$$
h^{(\beta)}(+\infty):=\lim_{s\to\infty}h^{(\beta)}(s)
$$
and
$$
\mathbb{D}^\beta h(t)=-\int_0^\infty h^{(\beta+1)}(t+s)  \, ds=-h^{(\beta)}(+\infty)+h^{(\beta)}(t).
$$
Therefore, if $h^{(\beta)}(+\infty)=0$,  $$\mathbb{D}^\beta h(t)=h^{(\beta)}(t).$$

Taking this into account and admitting a slight abuse of notation, we will denote the usual and the Weyl derivatives in the same way.

Our fist result concerns $L^p$-boundedness of variation, oscillation and jump operators for fractional derivatives of Fourier-Bessel Poisson semigroups.

\begin{Th}\label{Th:1.1}
Let $\beta \geq 0$, $\rho>2$, $\lambda>0$, $\nu > -1$
and $\{t_j\}_{j \in \NN}$ be a decreasing sequence in $(0,\infty)$.
Then, the operators
\begin{multicols}{2}
\begin{itemize}
\item[$i)$] $\mathcal{V}_\rho(\{t^\beta \partial_t^\beta P_t^\nu\}_{t>0})$,\\
\item[$ii)$] $\lambda [\Lambda(\{t^\beta \partial_t^\beta P_t^\nu\}_{t>0},\lambda)]^{1/p}$,
\end{itemize}
\begin{itemize}
\item[$iii)$] $\mathcal{O}(\{t^\beta \partial_t^\beta P_t^\nu\}_{t>0},\{t_j\}_{j \in \NN})$, \\
\item[$iv)$] $\mathcal{S}_V(\{t^\beta \partial_t^\beta P_t^\nu\}_{t>0})$,
\end{itemize}
\end{multicols}
are bounded from 
$L^p((0,1),x^{2 \nu + 1}dx)$ into itself, for
every $1<p<\infty$, and from $L^1((0,1),x^{2 \nu + 1}dx)$ into 
$L^{1,\infty}((0,1),x^{2 \nu + 1}dx)$.
\end{Th}

Hardy spaces $H^1((0,1),\Delta_\nu)$, for Fourier-Bessel expansions, have been considered in \cite{BDL, DPRS}. 
For a function $f \in L^1((0,1),x^{2 \nu + 1}dx)$, in 
\cite{DPRS} it is said that
$$f \in H^1((0,1),\Delta_\nu) 
\quad \text{provided that} \quad
P_*^\nu(f) \in L^1((0,1),x^{2 \nu + 1}dx),$$
and in \cite[Theorem 4.26]{BDL} it is proven that
$$f \in H^1((0,1),\Delta_\nu) 
\quad \text{if and only if} \quad
W_*^\nu(f) \in L^1((0,1),x^{2 \nu + 1}dx).$$
Moreover, in \cite[Theorem A]{DPRS} it was established an atomic characterization for $H^1((0,1),\Delta_\nu)$, for $\nu>-1/2$.\\

We characterize $H^1((0,1),\Delta_\nu)$ by using the variation operators associated to $\{P_t^\nu\}_{t>0}$.

\begin{Th}\label{Th:1.2}
Let $\rho>2$ and {$\nu>-1$}. Then, the operator
$\mathcal{V}_\rho(\{  P_t^\nu\}_{t>0})$ is bounded from
$H^1((0,1),\Delta_\nu)$ into $L^1((0,1),x^{2 \nu + 1}dx)$.
Furthermore, a function 
$f \in L^1((0,1),$ $x^{2 \nu + 1}dx)$ is in $H^1((0,1),\Delta_\nu)$
if and only if
$\mathcal{V}_\rho(\{  P_t^\nu\}_{t>0})(f) \in L^1((0,1),x^{2 \nu + 1}dx)$ and the quantities
\begin{equation}\label{eq:Th1.2a}
\|f\|_{L^1((0,1),x^{2 \nu + 1}dx)}
+ \|P_*^\nu f\|_{L^1((0,1),x^{2 \nu + 1}dx)}
\end{equation}
and
\begin{equation}\label{eq:Th1.2b}
\|f\|_{L^1((0,1),x^{2 \nu + 1}dx)}
+ \|\mathcal{V}_\rho(\{  P_t^\nu\}_{t>0})(f)\|_{L^1((0,1),x^{2 \nu + 1}dx)}
\end{equation}
are equivalent.
\end{Th}

It is also usual to study harmonic analysis associated with the Bessel type operator $\mathbb{S}_\nu$ defined by 
\begin{equation*}
\mathbb{S}_\nu
:=
- \frac{d^2}{dx^2}
+
\frac{\nu^2-1/4}{x^2} \quad
\text{in }
(0,1).
\end{equation*}
It is clear that
\begin{equation*}
\mathbb{S}_\nu
=
x^{\nu + 1/2} \Delta_\nu x^{-\nu-1/2}.
\end{equation*}
This relation allows us to transfer properties between $\Delta_\nu$ and $\mathbb{S}_\nu$ contexts.
However, the harmonic analysis related to 
$\Delta_\nu$ and $\mathbb{S}_\nu$ has important differences. For instance,
as it is shown in \cite{DPRS} (see also \cite{BDL})
the Hardy spaces 
$H^p((0,1),\Delta_\nu)$
and
$H^p((0,1),\mathbb{S}_\nu)$
have different properties.
Here we will prove that the variation operators defined by using semigroups associated with
$\Delta_\nu$ and $\mathbb{S}_\nu$ have different
$L^p$-boundedness properties.\\

For every $n \in \mathbb{N}$, we define
\begin{equation*}
\Psi_n^\nu(x)
:=
d_{n,\nu}
(\lambda_{n,\nu} x)^{1/2}
J_{\nu}(\lambda_{n,\nu} x), \quad x \in (0,1).
\end{equation*}
We have that 
\begin{equation*}
\mathbb{S}_\nu    
\Psi_n^\nu
=
\lambda_{n,\nu}^2
\Psi_n^\nu, \quad n \in \mathbb{N}.
\end{equation*}
The sequence
$\{\Psi_n^\nu\}_{n \in \mathbb{N}} $ 
is an orthonormal complete basis in 
$L^2((0,1),dx)$.
We define, for every 
$n \in \mathbb{N}$ and 
$f \in L^2((0,1),dx)$,
\begin{equation*}
a_n^\nu(f)
:= \int_0^1 
\Psi_n^\nu(x)
f(x)\, dx.
\end{equation*}
Moreover, consider the operator $S_\nu$ determined by
$$S_\nu f
:=
\sum_{n=1}^\infty 
\lambda_{n,\nu}^2 
a_n^\nu(f)
\Psi_n^\nu, 
\quad f \in D(S_\nu),$$
where
$$D(S_\nu)
:=
\Big\{ f \in L^2((0,1),dx)
\text{ : }
\sum_{n=1}^\infty 
|\lambda_{n,\nu}^2 
a_n^\nu(f)|^2 < \infty
\Big\}.$$
This operator $S_\nu$ is symmetric and positive in 
$L^2((0,1),dx)$. Note that when
$0<|\nu|<1$ the operators 
$\mathbb{S}_\nu$
and
$\mathbb{S}_{-\nu}$
are equal but the operators
${S}_\nu$
and
${S}_{-\nu}$
do not coincide.\\

Furthermore, $-S_\nu$ generates the semigroup of operators 
$\{\mathcal{W}_t^\nu\}_{t>0}$ in 
$L^2((0,1),dx)$, where
$$
\mathcal{W}_t^\nu(f)
:=
\sum_{n=1}^\infty
e^{-t \lambda_{n,\nu}^2} a_n^\nu(f) \Psi_n^\nu,
\quad t>0 
\quad \text{and} \quad
f \in L^2((0,1),dx).
$$
We can write,
for every $t>0$ and $f \in L^2((0,1),dx)$, 
\begin{equation*}
\mathcal{W}_t^\nu(f)
=
\int_0^1 \mathcal{W}_t^\nu(x,y) f(y) \, dy, \quad x\in (0,1),
\end{equation*}
with
$$
\mathcal{W}_t^\nu(x,y)
:=
\sum_{n=1}^\infty
e^{-t \lambda_{n,\nu}^2} \Psi_n^\nu(x) \Psi_n^\nu(y), \quad x,y\in (0,1).
$$

Since
\begin{equation*}
\lim_{z \to 0}
\frac{J_\nu(z)}{z^\nu} 
=
\frac{1}{2^\nu \Gamma(\nu+1)}
\end{equation*}
and
\begin{equation*}
|\sqrt{z} J_\nu(z)|
\lesssim 1, \quad z \in (1,\infty),
\end{equation*}
see \cite[formula (5.16.1)]{Leb}, if $n \in \mathbb{N}$ and $1<p<\infty$,
$\Psi_n^\nu \in L^p((0,1),dx)$
if and only if $\nu >-1/p-1/2$.
According to 
\cite[pp. 1339-1340]{NR1}
the maximal operator 
$\mathcal{W}_*^\nu$ defined by
\begin{equation*}
\mathcal{W}_*^\nu(f)
:= \sup_{t>0} |\mathcal{W}_t^\nu(f)|,
\end{equation*}
is 
\begin{itemize}
\item bounded from $L^p((0,1),dx)$ into itself when $1<p<\infty$ and
$-\nu-1/2<1/p<\nu + 3/2$; \\
\item of restricted weak type $(p,p)$ when $1<p<\infty$, 
$-1< \nu {<} -1/2$ and $p=1/(\nu+3/2)$;\\
\item of weak type $(1,1)$ when $\nu {\geq}-1/2$.\\
\end{itemize} 

The Poisson semigroup 
$\{\mathcal{P}_t^\nu\}_{t>0}$ generated by $-\sqrt{S_\nu}$ is defined by
\begin{equation*}
\mathcal{P}_t^\nu(f)
:=
\sum_{n=1}^\infty e^{-\lambda_{n,\nu}t}
a_n^\nu(f) \Psi_n^\nu, \quad
f \in L^2((0,1),dx) \quad \text{and} \quad t>0.
\end{equation*}
Also, by using subordination formula we can write
\begin{equation*}
\mathcal{P}_t^\nu(f)
=
\frac{t}{2 \sqrt{\pi}}
\int_0^\infty \frac{e^{-t^2/(4u)}}{u^{3/2}}
\mathcal{W}_u^\nu(f) \, du
, \quad
f \in L^2((0,1),dx) \quad \text{and} \quad t>0.
\end{equation*}
Then, $L^p$-boundedness properties of the maximal operator $\mathcal{P}_*^\nu$ defined by
\begin{equation*}
\mathcal{P}_*^\nu(f)
:= \sup_{t>0} |\mathcal{P}_t^\nu(f)|,
\end{equation*}
can be deduced as the corresponding properties of $\mathcal{W}_*^\nu$ that we have just stated above.\\

$L^p$-boundedness properties of variation, oscillation and jump operators for 
$\{t^\beta \partial_t^\beta \mathcal{P}_t^\nu\}_{t>0}$
are established in the following

\begin{Th}\label{Th:1.3}
Let $\beta \geq 0$, $\rho>2$, $\lambda>0$, $\nu > -1$
and $\{t_j\}_{j \in \NN}$ be a decreasing sequence in $(0,\infty)$.
Then, the operators
\begin{multicols}{2}
\begin{itemize}
\item[$i)$] $\mathcal{V}_\rho(\{t^\beta \partial_t^\beta \mathcal{P}_t^\nu\}_{t>0})$,\\
\item[$ii)$] $\lambda [\Lambda(\{t^\beta \partial_t^\beta \mathcal{P}_t^\nu\}_{t>0},\lambda)]^{1/p}$,
\end{itemize}
\begin{itemize}
\item[$iii)$] $\mathcal{O}(\{t^\beta \partial_t^\beta \mathcal{P}_t^\nu\}_{t>0},\{t_j\}_{j \in \NN})$, \\
\item[$iv)$] $\mathcal{S}_V(\{t^\beta \partial_t^\beta \mathcal{P}_t^\nu\}_{t>0})$,
\end{itemize}
\end{multicols}
are
\begin{itemize}
\item[$(a)$] of strong type $(p,p)$ when $1<p<\infty$ and 
$-\nu-1/2<1/p<\nu + 3/2$;\\
\item[$(b)$] of weak type $(1,1)$
when $\nu {\geq} -1/2$;\\
\item[$(c)$] of restricted weak type $(p,p)$ when $1<p<\infty$, 
$-1< \nu < -1/2$ and $p=-1/(\nu+1/2)$
or $p=1/(\nu+3/2)$.
\end{itemize}
\end{Th}

Hardy spaces associated to the operator $S_\nu$ were studied in 
\cite{BDL,DPRS}. Let 
$f \in L^1((0,1),dx)$. We say that
(\cite[\S 1.2]{DPRS})
\begin{equation*}
f \in H^1((0,1),S_\nu)
\quad \text{when} \quad
\mathcal{P}_*^\nu(f) \in L^1((0,1),dx).    
\end{equation*}
In \cite[Theorem 4.1.7]{BDL} it was proven that, for $\nu>-1/2$,
\begin{equation*}
f \in H^1((0,1),S_\nu)
\quad \text{if and only if} \quad \mathcal{W}_*^\nu(f) \in L^1((0,1),dx).   
\end{equation*}

In the following we characterize 
$H^1((0,1),S_\nu)$ by using variation operators associated with 
$\{ \mathcal{P}_t^\nu\}_{t>0}$.

\begin{Th}\label{Th:1.4}
Let $\rho>2$ and ${\nu>-1/2}$. Then, the operator
$\mathcal{V}_\rho(\{\mathcal{P}_t^\nu\}_{t>0})$ is bounded from
$H^1((0,1),S_\nu)$ into $L^1((0,1),dx)$.
Furthermore, if 
$f \in L^1((0,1),dx)$ then 
$f \in H^1((0,1),S_\nu)$
if and only if
$\mathcal{V}_\rho(\{  \mathcal{P}_t^\nu\}_{t>0}) \in L^1((0,1),dx)$ and the quantities
$$
\|f\|_{L^1((0,1),dx)}
+ \|\mathcal{P}_*^\nu f\|_{L^1((0,1),dx)}
$$
and
$$
\|f\|_{L^1((0,1),dx)}
+ \|\mathcal{V}_\rho(\{  \mathcal{P}_t^\nu\}_{t>0})(f)\|_{L^1((0,1),dx)}
$$
are equivalent.
\end{Th}

The paper is organized as follows. 
In Sections \ref{Sect2} and \ref{Sect3} we prove Theorems \ref{Th:1.1} and \ref{Th:1.3}, respectively. 
Theorems \ref{Th:1.2} and \ref{Th:1.4} are proven in Sections \ref{S4} and \ref{Sect5}.\\

Throughout this article $C$ and $c$ always denote positive constants that can change in each occurrence. We also write $a\lesssim b$ as shorthand for $a\leq Cb$ and moreover will use the notation $a\sim b$ if $a\lesssim b$ and $b\lesssim a$.

\section{Proof of Theorem \ref{Th:1.1}}
\label{Sect2}

\subsection{The variation operator 
$\mathcal{V}_\rho(\{t^\beta \partial_t^\beta P_t^\nu\}_{t>0})$}
\quad \\

Suppose that $f(x)=1$, $x \in (0,\infty)$. By using \cite[Proposition 2.3]{Stem} we deduce that the series
\begin{equation*}
\sum_{n=1}^\infty e^{-\lambda_{n,\nu}t}
\phi_n^\nu(x)c_n^\nu(f)
\end{equation*}
converges uniformly in $(0,1)$ and it can be extended to $[0,1]$ as a continuous function, for every $t>0$. 
Furthermore, since $\phi_n^\nu(1)=0$, $n \in \mathbb{N}$, we deduce that $P_t^\nu(f) \neq f$, $t>0$. Hence the semigroup of operators $\{P_t^\nu\}_{t>0}$ is not Markovian.\\

For every $t>0$, $P_t^\nu$ is selfadjoint in $L^2((0,1),x^{2\nu+1}dx)$.
Then, according to \cite[III.2]{Yo},
$\{P_t^\nu\}_{t>0}$ is a bounded analytic semigroup on $L^2((0,1),x^{2\nu+1}dx)$.
At this moment, it is not clear for us if, for every $t>0$, $P_t^\nu$ is an absolute contraction (see \cite[(6.1)]{LeMX}).\\

Thus, $\{P_t^\nu\}_{t>0}$ is not a symmetric diffusion semigroup and we do not know if 
\cite[Corollary 6.1]{LeMX} can be applied to $\{P_t^\nu\}_{t>0}$. 
The same arguments are valid for the semigroup $\{W_t^\nu\}_{t>0}$ of operators.
Note that \cite[Corollary 6.1]{LeMX}
is concerned with the strong type $(p,p)$ with $1<p<\infty$ for the variation operator 
$\mathcal{V}_\rho(\{t^m \partial_t^m T_t\}_{t>0})$ where $m \in \mathbb{N}$ and $\{T_t\}_{t>0}$ is a bounded analytic semigroup of operators and each $T_t$ is an absolute contraction.\\
\begin{Rem}\label{derivada}
Let $\beta\ge 0$, $m \in \NN$ such that $m-1 \leq \beta < m$ and  $f\in L^p((0,1), x^{2\nu+1}dx)$.
Then,  we have that
\begin{equation*}
 \partial_t^\beta P_t^\nu(f)(x)
=
\frac{-1}{\Gamma(m-\beta)}
\int_0^\infty \partial_t^m P_{t+s}^\nu(f)(x)
s^{m-\beta-1} \, ds,   \quad x\in (0,1),\:\;  t>0.
\end{equation*}
Let $x\in (0,1)$. By using the subordination formula we can write
\begin{align}\label{eq:I1}
\partial_t^m P_{t+s}^\nu(f)(x)    
&=
\partial_t^m 
\Big[
\frac{t+s}{2\sqrt{\pi}} \int_0^\infty 
\frac{e^{-(t+s)^2/(4u)}}{u^{3/2}} W_u^\nu(f)(x) \, du
\Big] \nonumber \\
& = 
\frac{1}{\sqrt{\pi}}
\int_0^\infty 
\partial_t^{m+1} 
\Big[-
e^{-(t+s)^2/(4u)}
\Big]
\frac{W_u^\nu(f)(x)}{\sqrt{u}} \, du,
 \quad t,s>0.
\end{align}
Indeed, by \cite[Lemma 3]{BCCFR}, for every $k \in \NN$, we get
\begin{equation*}
\Big| 
\partial_t^k [e^{-(t+s)^2/(4u)}]
\Big|
\lesssim
\frac{e^{-c(t+s)^2/u}}{u^{k/2}}, \quad
s,t,u>0,
\end{equation*}
and according to \cite[Theorem 1]{MSZ} (see also 
 \cite[Theorem A]{NR2}), for $x,y \in (0,1)$ and $t>0$,
\begin{equation}\label{eq:2.1}
W_t^\nu(x,y)
\sim
\frac{(1+t)^{\nu+2}}{(t+xy)^{\nu+1/2}}
\Big( 1 \wedge 
\frac{(1-x)(1-y)}{t}\Big)
\frac{1}{\sqrt{t}}
\exp\Big( - \frac{|x-y|^2}{4t} - \lambda_{1,\nu}^2 t\Big),
\end{equation}
where $a \land b := \min\{a,b\}$, $a, b \in \mathbb{R}$.
Hence,  for $\nu>-1$  and $t>M$,
\begin{align*}
& \int_0^\infty\int_0^1 \frac{|\partial_t^{m+1}e^{-t^2/(4u)}|}{u^{1/2}}W_u^\nu (x,y)|f(y)|y^{2\nu+1}dy\;du \\
& \quad \lesssim 
\int_0^\infty \int_0^1
\frac{e^{-ct^2/u}}{u^{(m+2)/2}}
W_u^\nu(x,y) \, du |f(y)| y^{2\nu+1} \, dy\\
&\quad \lesssim 
\int_0^\infty \int_0^1 \frac{e^{-ct^2/u}}{u^{(m+2)/2}}\frac{(1+u)^{\nu+2}}{(u+xy)^{\nu+1/2}}
\Big( 1 \wedge 
\frac{(1-x)(1-y)}{u}\Big)
\frac{1}{\sqrt{u}}
\exp\Big( - \frac{|x-y|^2}{4u} - \lambda_{1,\nu}^2 u\Big) \, du\\
& \qquad \times |f(y)|y^{2\nu+1}dy\\
& \quad \lesssim
\left\{
\begin{array}{ll}
   \displaystyle \int_0^\infty \frac{e^{-ct^2/u}}{u^{(m+3)/2+\nu+1/2}}{(1+u)^{\nu+2}}
e^{- \lambda_{1,\nu}^2 u}{du}\int_0^1|f(y)|y^{2\nu+1}dy, & \nu\ge -1/2  \\
& \\
 \displaystyle\int_0^\infty \frac{e^{-ct^2/u}}{u^{(m+3)/2}}{(1+u)^{\nu+2}}e^{- \lambda_{1,\nu}^2 u} & \\
 \qquad \times \displaystyle \int_0^1{(u^{-\nu-1/2}+(xy)^{-\nu-1/2})}
|f(y)|y^{2\nu+1}dy\;du,&    -1<\nu<-1/2
\end{array}
\right.\\
&\quad \lesssim
\left\{
\begin{array}{ll}
   \displaystyle \int_0^\infty \frac{e^{-cM^2/u}}{u^{(m+3)/2+\nu+1/2}}{(1+u)^{\nu+2}}
e^{- \lambda_{1,\nu}^2 u}{du}\|f\|_{L^p((0,1), x^{2\nu+1}dx)}, & \nu\ge -1/2  \\
& \\
 \displaystyle\int_0^\infty \frac{e^{-cM^2/u}}{u^{(m+3)/2}}{(1+u)^{\nu+2}}(u^{-\nu-1/2}+1)e^{- \lambda_{1,\nu}^2 u} \, du \|f\|_{L^p((0,1), x^{2\nu+1}dx)},   & -1<\nu<-1/2
\end{array}
\right.
\\
& \quad <\infty.
\end{align*}
Therefore, the derivation under the integral sign in \eqref{eq:I1}
is justified. 
Moreover,  Dominated Convergence Theorem gives that
$$
\lim_{t\to\infty} \partial_t^\beta (P_t^\nu (f)(x))=0.
$$
Therefore, if $\beta\in\NN_0,$
$$
 \partial_t^\beta (P_t^\nu (f)(x))=(P_t^\nu(f)(x))^{(\beta)}, \:\:t>0.
$$
\end{Rem}

Next lemma establishes a useful pointwise connection between the variation operators associated to 
$\{t^\beta \partial_t^\beta P_t^\nu\}_{t>0}$
and
$\{W_t^\nu\}_{t>0}$.

\begin{Lem}\label{Lem:2.4}
Let $\nu>-1$, $\rho>2$, $\beta \geq 0$ and $f\in L^p((0,1), x^{2\nu+1}dx)$. 
Then,
\begin{align}\label{eq:I2}
\mathcal{V}_\rho(\{t^\beta \partial_t^\beta P_t^\nu\}_{t>0})(f)(x)
& \lesssim
\mathcal{V}_\rho(\{W_t^\nu\}_{t>0})(f)(x), 
\quad x \in (0,1).
\end{align}
\end{Lem}

\begin{proof}
Fix $x \in (0,1)$.
First of all, if $\beta=0$, the estimate \eqref{eq:I2} is an straightforward consequence of the subordination formula \eqref{eq:subordination}. 

For $t,s,u>0$, Fa di Bruno's formula leads to
\begin{align*}
\partial_t^{m+1} [e^{-(t+s)^2/(4u)}]    
&=
\sum_{k_1+2k_2=m+1}
\frac{(m+1)!(-1)^{k_1+k_2}}{k_1! k_2! 2!^{k_2}}
e^{-(t+s)^2/(4u)}
\Big( \frac{t+s}{2u}\Big)^{k_1}
\Big( \frac{1}{2u}\Big)^{k_2}.
\end{align*}
Therefore,
\begin{align*}
t^\beta \partial_t^\beta P_{t}^\nu(f)(x) 
& = 
\frac{-(m+1)!}{\Gamma(m-\beta) \sqrt{\pi}}
\sum_{k_1+2k_2=m+1}
\frac{(-1)^{k_1+k_2}}{k_1! k_2! 2!^{k_2}}
I_{k_1, k_2}(f)(t,x), \quad t>0,
\end{align*}
where, for every $k_1$, $k_2 \in \NN$ such that
$k_1+2k_2=m+1$ and  $t>0$,
\begin{equation*}
I_{k_1, k_2}(f)(t,x)
:=
t^\beta\int_0^\infty
s^{m-\beta-1}
\int_0^\infty 
\frac{e^{-(t+s)^2/(4u)}}{(2u)^{k_1+k_2} \sqrt{u}}
(t+s)^{k_1} W_u^\nu(f)(x) \, du \, ds.
\end{equation*}
Let $k_1$, $k_2 \in \NN$ such that
$k_1+2k_2=m+1$. The change of variables
$v=(t+s)^2/(4u)$ and $s=tz$, yield to
\begin{align*}
I_{k_1, k_2}(f)(t,x)
& =
2^{k_1+k_2-1} t^\beta
\int_0^\infty
\frac{s^{m-\beta-1}}{(t+s)^m}
\int_0^\infty e^{-v} v^{k_1+k_2-3/2}
W_{(t+s)^2/(4v)}^\nu (f)(x) \, dv \, ds \\
& =
2^{k_1+k_2-1}   
\int_0^\infty
\frac{z^{m-\beta-1}}{(1+z)^m}
\int_0^\infty e^{-v} v^{k_1+k_2-3/2}
W_{t^2(1+z)^2/(4v)}^\nu (f)(x) \, dv \, dz.
\end{align*}
For  $t>0$, we conclude
\begin{align}\label{eq:I3}
t^\beta \partial_t^\beta P_{t}^\nu(f)(x) 
& =
\frac{-(m+1)!}{\Gamma(m-\beta) \sqrt{\pi}}
\sum_{k_1+2k_2=m+1}
\frac{(-1)^{k_1+k_2}2^{k_!+k_2-1}}{k_1! k_2! 2!^{k_2}} \nonumber \\
& \qquad \times 
\int_0^\infty
\frac{z^{m-\beta-1}}{(1+z)^m}
\int_0^\infty e^{-v} v^{k_1+k_2-3/2}
W_{t^2(1+z)^2/(4v)}^\nu (f)(x) \, dv \, dz.
\end{align}
On the other hand, if $\{t_j\}_{j \in \mathbb{N}}$ is a decreasing sequence in $(0,\infty)$ and $n \in \NN$,
Minkowski's inequality gives us
\begin{align*}
& \Big( 
\sum_{j=1}^n
\Big| 
t^\beta \partial_t^\beta P_t^\nu(f)(x)_{|_{t=t_j}}
-
t^\beta \partial_t^\beta P_t^\nu(f)(x)_{|_{t=t_{j+1}}}
\Big|^\rho
\Big)^{1/\rho}    \\
& \qquad
\lesssim
\sum_{k_1+2k_2=m+1}
\int_0^\infty
\frac{z^{m-\beta-1}}{(1+z)^m}
\int_0^\infty e^{-v} v^{k_1+k_2-3/2} \\
& \hspace{3.5cm} \times 
\Big(
\sum_{j=1}^n
\Big|
W_{t_j^2(1+z)^2/(4v)}^\nu (f)(x) 
-
W_{t_{j+1}^2(1+z)^2/(4v)}^\nu (f)(x) 
\Big|^\rho
\Big)^{1/\rho}
\, dv \, dz.
\end{align*}
Then,
\begin{align*} 
\mathcal{V}_\rho(\{t^\beta \partial_t^\beta P_t^\nu\}_{t>0})(f)(x)
& \lesssim
\sum_{k_1+2k_2=m+1}
\int_0^\infty
\frac{z^{m-\beta-1}}{(1+z)^m}
\int_0^\infty e^{-v} v^{k_1+k_2-3/2} \, dv \, dz \nonumber \\
& \quad \times
\mathcal{V}_\rho(\{W_t^\nu\}_{t>0})(f)(x) \nonumber \\
& \lesssim
\mathcal{V}_\rho(\{W_t^\nu\}_{t>0})(f)(x).
\end{align*}
Note that if $k_1$, $k_2 \in \NN$ such that
$k_1+2k_2=m+1$ we have that 
\begin{equation*}
k_1 + k_2 - \frac{3}{2}
=
m-k_2 - \frac{1}{2}
\geq
m - \frac{m+1}{2} - \frac{1}{2}
=
\frac{m}{2}-1. \qedhere
\end{equation*}
\end{proof}

As it was mentioned above, the semigroup 
$\{W_t^\nu\}_{t>0}$ is not a diffusion semigroup and then 
\cite[Corollary 6.1]{LeMX} cannot be used to obtain
$L^p$-boundedness properties for 
$\mathcal{V}_\rho(\{W_t^\nu\}_{t>0})$.
We continue analyzing the following $L^2$-boundedness.

\begin{Prop}\label{Prop:2.3}
Let $\nu>-1$ and $\rho>2$. The operator 
$\mathcal{V}_\rho(\{W_t^\nu\}_{t>0})$ is bounded from $L^2((0,1),x^{2\nu+1}dx)$ into itself.
\end{Prop}

\begin{proof}
For every $n \in \NN$, we define
$
\varphi_n^\nu 
:= \phi_n^\nu/\phi_1^\nu.   
$
The sequence $\{\varphi_n^\nu \}_{n \in \NN}$ is a complete orthonormal basis in 
$L^2\Big((0,1),x^{2\nu+1}(\phi_1^\nu(x))^2dx\Big)$.
We consider, for every 
$f \in L^2\Big((0,1),x^{2\nu+1}(\phi_1^\nu(x))^2dx\Big)$ and $t>0$, 
\begin{equation*}
\WW_t^\nu(f)
:=
\sum_{n=1}^\infty 
e^{-(\lambda_{n,\nu}^2-\lambda_{1,\nu}^2)t}
d_n^\nu(f) \varphi_n^\nu,
\end{equation*}
where, for every $n \in \NN$,
\begin{equation*}
d_n^\nu(f)   
:=
\int_0^1 f(y) \varphi_n^\nu(y) y^{2\nu+1} 
(\phi_1^\nu(y))^2 \, dy.
\end{equation*}

The semigroup $\{\WW_t^\nu\}_{t>0}$ was studied in \cite{LN}.
$\{\WW_t^\nu\}_{t>0}$ is a symmetric diffusion semigroup with respect to 
$\Big((0,1),x^{2\nu+1}(\phi_1^\nu(x))^2dx\Big)$.
Then, according to \cite[Corollary 6.1]{LeMX}
the operator 
$\mathcal{V}_\rho(\{\WW_t^\nu\}_{t>0})$ 
is bounded from
$L^p\Big((0,1),x^{2\nu+1}(\phi_1^\nu(x))^2dx\Big)$
into itself, for every $1<p<\infty$.\\

We define, for every $t>0$,
\begin{equation*}
\widetilde{W}_t^\nu
:=
e^{\lambda_{1,\nu}^2 t} W_t^\nu.
\end{equation*}
We have that
\begin{equation*}
\WW_t^\nu(f)
=
\frac{1}{\phi_1^\nu} 
\widetilde{W}_t^\nu(f \phi_1^\nu), \quad t>0.
\end{equation*}
Then,
\begin{equation*}
\mathcal{V}_\rho(\{\WW_t^\nu\}_{t>0})(f)
=
\frac{1}{\phi_1^\nu} 
\mathcal{V}_\rho(\{\widetilde{W}_t^\nu\}_{t>0})(f \phi_1^\nu).
\end{equation*}
It follows that 
$\mathcal{V}_\rho(\{\widetilde{W}_t^\nu\}_{t>0})$
is bounded from 
$L^2((0,1),x^{2\nu+1}dx)$ into itself.\\

Suppose that $\{t_j\}_{j \in \NN}$ is a decreasing sequence in $(0,\infty)$. We can write,
for $x \in (0,1)$,
\begin{align*}
&
\Big(
\sum_{j=1}^\infty
\Big|
W_{t_j}^\nu (f)(x) 
-
W_{t_{j+1}}^\nu (f)(x) 
\Big|^\rho
\Big)^{1/\rho}  \\
& \qquad 
=
\Big(
\sum_{j=1}^\infty
\Big|
e^{-\lambda_{1,\nu}^2 t_j}
\widetilde{W}_{t_j}^\nu (f)(x) 
-
e^{-\lambda_{1,\nu}^2 t_{j+1}}
\widetilde{W}_{t_{j+1}}^\nu (f)(x) 
\Big|^\rho
\Big)^{1/\rho}  \\
& \qquad 
\leq 
\Big(
\sum_{j=1}^\infty
e^{-\lambda_{1,\nu}^2 t_j \rho}
\Big|
\widetilde{W}_{t_j}^\nu (f)(x) 
-
\widetilde{W}_{t_{j+1}}^\nu (f)(x) 
\Big|^\rho
\Big)^{1/\rho} \\
& \qquad \quad 
+
\Big(
\sum_{j=1}^\infty
\Big|
e^{-\lambda_{1,\nu}^2 t_j}
-
e^{-\lambda_{1,\nu}^2 t_{j+1}}
\Big|^\rho
\Big|\widetilde{W}_{t_{j+1}}^\nu (f)(x) \Big|^\rho
\Big)^{1/\rho} \\
& \qquad 
\leq 
\Big(
\sum_{j=1}^\infty
\Big|
\widetilde{W}_{t_j}^\nu (f)(x) 
-
\widetilde{W}_{t_{j+1}}^\nu (f)(x) 
\Big|^\rho
\Big)^{1/\rho} 
+
\sup_{t>0}
\Big|\widetilde{W}_{t}^\nu (f)(x) \Big|
\,
\sum_{j=1}^\infty
\Big|
e^{-\lambda_{1,\nu}^2 t_j}
-
e^{-\lambda_{1,\nu}^2 t_{j+1}}
\Big| \\
& \qquad 
\lesssim 
\Big(
\sum_{j=1}^\infty
\Big|
\widetilde{W}_{t_j}^\nu (f)(x) 
-
\widetilde{W}_{t_{j+1}}^\nu (f)(x) 
\Big|^\rho
\Big)^{1/\rho} 
+
\sup_{t>0}
\Big|\widetilde{W}_{t}^\nu (f)(x) \Big|
\,
\sum_{j=1}^\infty
\int_{t_{j+1}}^{t_j}
e^{-\lambda_{1,\nu}^2 t}
\, dt \\
& \qquad 
\lesssim
\Big(
\sum_{j=1}^\infty
\Big|
\widetilde{W}_{t_j}^\nu (f)(x) 
-
\widetilde{W}_{t_{j+1}}^\nu (f)(x) 
\Big|^\rho
\Big)^{1/\rho} 
+
\sup_{t>0}
\Big|\widetilde{W}_{t}^\nu (f)(x) \Big|
\,
\int_{0}^{\infty}
e^{-\lambda_{1,\nu}^2 t}
\, dt.
\end{align*}
Then,
\begin{equation*}
\mathcal{V}_\rho(\{W_t^\nu\}_{t>0})(f)
\lesssim
\mathcal{V}_\rho(\{\widetilde{W}_t^\nu\}_{t>0})(f)
+
 \widetilde{W}_*^\nu(f),
\end{equation*}
where
\begin{equation*}
\widetilde{W}_*^\nu(f)    
:= \sup_{t>0} |\widetilde{W}_t^\nu(f)|.
\end{equation*}
According to \cite[p. 73]{SteinLP}
the maximal operator $\WW_*^\nu$ defined by
\begin{equation*}
\WW_*^\nu(f)    
:= \sup_{t>0} |\WW_t^\nu(f)|.
\end{equation*}
is bounded from 
$L^p\Big((0,1),x^{2\nu+1}(\phi_1^\nu(x))^2dx\Big)$
into itself, for every $1<p<\infty$. Then, $\widetilde{W}_*$ is bounded from
$L^2((0,1),x^{2\nu+1}dx)$ into itself.\\

We conclude that the operator 
$\mathcal{V}_\rho(\{W_t^\nu\}_{t>0})$ is bounded from $L^2((0,1),x^{2\nu+1}dx)$ into itself.
\end{proof}

As a consequence of Lemma \ref{Lem:2.4}
and Proposition \ref{Prop:2.3} we obtain the following result.

\begin{Cor}
Let $\nu>-1$, $\rho>2$ and $\beta \geq 0$. The operator 
$\mathcal{V}_\rho(\{t^\beta \partial_t^\beta P_t^\nu\}_{t>0})$ is bounded from $L^2((0,1),x^{2\nu+1}dx)$ into itself.
\end{Cor}

$L^p$-boundedness properties of 
$\mathcal{V}_\rho(\{ \WW_t^\nu\}_{t>0})$ cannot be transferred to 
$\mathcal{V}_\rho(\{W_t^\nu\}_{t>0})$
when $p \neq 2$, as we have just done for $p=2$ in the proof of Proposition \ref{Prop:2.3}.
Instead, we are going to use vector-valued Calder\'on-Zygmund theory
(see \cite{RuFRT1,RuFRT2}) as explained below.\\

We define $E_\rho$ as the space of continuous functions $\Psi$ on $(0,\infty)$ such that
\begin{equation*}
\|\Psi\|_\rho    
:=
\sup_{\substack{0<t_n<t_{n-1}<\dots<t_1\\n\in\NN}}
\Big( 
\sum_{j=1}^{n-1}
\Big| 
\Psi(t_j) - \Psi(t_{j+1})
\Big|^\rho 
\Big)^{1/\rho}
< \infty.
\end{equation*}
It is clear that, $\|\Psi\|_\rho=0$ if and only if
$\Psi(t)=\alpha$, $t \in (0,\infty)$, for some $\alpha \in \CC$. We say that $f \sim g$, where $f,g \in E_\rho$, when $f-g=\alpha$ on $(0,\infty)$ for some $\alpha \in \CC$. Then, by defining $\| \cdot \|_\rho$
on $E_\rho / \sim$ in the natural way,
$(E_\rho / \sim,\|\cdot\|_\rho)$ is a Banach space.\\

Suppose that $\Psi$ is a differentiable function on $(0,\infty)$. If $\{t_j\}_{j=1}^\infty$ is a decreasing sequence on $(0,\infty),$ we can write 
\begin{align*}
\left(\sum_{j=1}^\infty|\Psi(t_j)-\Psi(t_{j+1})|^\rho\right)^{1/\rho}&=\left(\sum_{j=1}^\infty\left|\int_{t_{j+1}}^{t_j}\Psi'(t)dt\right|^\rho\right)^{1/\rho}\le \sum_{j=1}^\infty\int_{t_{j+1}}^{t_j}\left|\Psi'(t)\right| dt\\
&\le\int_{0}^\infty\left|\Psi'(t)\right| dt.
\end{align*}
Therefore, 
\begin{equation}\label{A}
\mathcal{V}_\rho (\{g(t)\}_{t>0})\le \int_0^\infty |g'(t)|dt.
\end{equation}
We have that
\begin{equation*}
\mathcal{V}_\rho(\{t^\beta \partial_t^\beta P_t^\nu\}_{t>0})(f)(x)
=
\| t^\beta \partial_t^\beta P_t^\nu(f)(x) \|_{E_\rho},
\quad x \in (0,1).
\end{equation*}
We consider the operator $T_\beta^\nu$ defined by
\begin{equation*}
T_\beta^\nu(f)(t,x)
:=
t^\beta \partial_t^\beta P_t^\nu(f)(x),
\quad x \in (0,1), \quad t>0.
\end{equation*}
Also, we define
\begin{equation*}
\mathbb{T}_\beta^\nu(f)(x)
:=
\int_0^1 H_\beta^\nu(x,y) f(y) y^{2\nu+1} \, dy,
\end{equation*}
where, for every $x,y \in (0,1)$, $x \neq y$,
\begin{equation*}
\begin{array}{rrccl}
H_\beta^\nu(x,y) & : & (0,\infty) & \longrightarrow & \mathbb{R}   \\
&& t & \longmapsto &
[H_\beta^\nu(x,y)](t)
:= t^\beta \partial_t^\beta P_t^\nu(x,y),
\end{array}
\end{equation*}
and the integral is understood in the 
$E_\rho/\sim$ B\"ochner sense.\\
We are going to see that the operator $T_\beta^\nu$ is a $E_\rho$-Calder\'on-Zygmund operator associated with the kernel $H_\beta^\nu (x,y).$

In the next lemmas we prove size and regularity properties for $H_\beta^\nu$ establishing that it is a standard $E_\rho$-Calder\'on-Zygmund kernel.

\begin{Lem}\label{Prop:2.1}
Let $\nu>-1$ and $\beta\ge 0$. We have that, for any $x,y \in (0,1)$,
\begin{equation*}
\|
H_\beta^\nu(x,y)
\|_{E_\rho} 
\lesssim
\left\{
\begin{array}{ll}
x^{-2(\nu+1)},  
& 0 < y \leq x/2,  \\
& \\
(xy)^{-\nu-1/2}|x-y|^{-1}, 
& x/2 < y \leq \min\{1,3x/2\},\: \:x\neq y \\
& \\
y^{-2(\nu+1)}, & \min\{1,3x/2\} < y \leq 1.
\end{array}
\right.
\end{equation*}
\end{Lem}

\begin{proof}

Let $m\in\NN$. We have that
\begin{align*}
    \partial_t^mP_t^\nu(x,y)&=\partial_t^m\left( \frac{t}{2\sqrt{\pi}}\int_0^\infty \frac{e^{-\frac{t^2}{4s}}}{s^{3/2}}W_s^\nu (x,y)ds\right)\\
    &=\frac{-1}{\sqrt{\pi}}\int_0^\infty \frac{\partial_t^{m+1}(e^{-\frac{t^2}{4s}})}{s^{1/2}}W_s^\nu (x,y)ds,\: \:\; x,y\in (0,1) \text{ and } \:t>0.
    \end{align*}
The differentiation under the integral sign is justified since from \cite[Lemma 3]{BCCFR} and \eqref{eq:2.1} we get that 
 \begin{align*}
    \int_0^\infty &\frac{|\partial_t^{m+1}e^{-\frac{t^2}{4s}}|}{s^{1/2}}W_s^\nu (x,y)ds \lesssim \int_0^\infty e^{-ct^2/s}W_s^\nu(x,y)\frac{ds}{s^{\frac{m+2}{2}}}\\
    &\lesssim \int_0^\infty \frac{e^{-ct^2/s}}{s^{\frac{m+2}{2}}}\frac{(1+s)^{\nu+2}}{(s+xy)^{\nu+1/2}}
\Big( 1 \wedge 
\frac{(1-x)(1-y)}{s}\Big)
\frac{1}{\sqrt{s}}
\exp\Big( - \frac{|x-y|^2}{4s} - \lambda_{1,\nu}^2 s\Big){ds}<\infty.
\end{align*}
  \\

{Let $\beta\ge 0$, and} $m\in\NN$  such that $m-1\le\beta<m$. We have that
\begin{align*}
    \partial_t(t^\beta\partial_t^\beta P_t^\nu(x,y))&=\partial_t\left(\frac{- t^{\beta}}{\Gamma(m-\beta)}\int_0^\infty \partial_t^m P_{t+s}^\nu(x,y)s^{m-\beta-1}ds \right)\\
    &=\frac{-\beta t^{\beta-1}}{\Gamma(m-\beta)}\int_0^\infty \partial_t^m P_{t+s}^\nu(x,y)s^{m-\beta-1}ds\\
    & \qquad 
    +\frac{- t^{\beta}}{\Gamma(m-\beta)}\int_0^\infty \partial_t^{m+1} P_{t+s}^\nu(x,y)s^{m-\beta-1}ds\\
    &=\beta t^{\beta-1}\partial_t^\beta P_{t}^\nu(x,y)+t^{\beta}\partial_t^{\beta+1} P_{t}^\nu(x,y), \:\:x,y\in (0,1), \:\;x\neq y.
\end{align*}
The derivation under the integral sign is justified because, as it can be seen in the sequel, the integrals are absolutely convergent.
\\
Let  $x,y \in (0,1)$, $x\neq y$. According to \eqref{A},  we can write
\begin{align*}
    \|
H_\beta^\nu(x,y)
\|_{E_\rho}& \lesssim \int_0^\infty|\partial_t[H_\beta^\nu(x,y)](t)|dt\\
&\le \beta\int_0^\infty | t^{\beta-1}\partial_t^\beta P_{t}^\nu(x,y)|dt+\int_0^\infty |t^{\beta}\partial_t^{\beta+1} P_{t}^\nu(x,y)|dt\\
&\le \beta\|
H_\beta^\nu(x,y)
\|_{L^1((0,\infty),\frac{dt}{t})} +\|
H_{\beta+1}^\nu(x,y)
\|_{L^1((0,\infty),\frac{dt}{t})} 
\end{align*}

{Let $\gamma>0$, and} $m \in \mathbb{N}$ such that $m-1 \leq \gamma<m$.
    Then,
\begin{align*}
\|
H_\gamma^\nu(x,y)
\|_{L^1((0,\infty),\frac{dt}{t})} 
& \leq 
\frac{1}{\Gamma(m-\gamma)}
\int_0^\infty 
t^{\gamma-1}
\int_0^\infty 
|\partial_t^m P_{t+s}^\nu (x,y)|
s^{m-\gamma-1} \, ds \, dt \\
& = 
\frac{1}{\Gamma(m-\gamma)}
\int_0^\infty 
|\partial_u^m P_{u}^\nu (x,y)|
\int_0^u 
(u-t)^{m-\gamma-1} t^{\gamma-1} \, dt \, du \\
& \lesssim
\int_0^\infty 
u^{m-1}
|\partial_u^m P_{u}^\nu (x,y)|
 \, du.
\end{align*}

According to \cite[Lemma 3]{BCCFR}
 we deduce 
\begin{align}\label{eq:2.1.1}
|\partial_u^m P_{u}^\nu (x,y)|
& \lesssim
\int_0^\infty 
\frac{|\partial_u^m [u e^{-u^2/(4s)}]|}{s^{3/2}} |W_s^\nu(x,y)| \, ds \nonumber \\
& \lesssim
\int_0^\infty 
\frac{e^{-c u^2/s}}{s^{(m+2)/2}} |W_s^\nu(x,y)| \, ds,
\quad u>0.
\end{align} 
Moreover, if we define
\begin{equation*}
K_m^\nu(x,y)
:=
\int_0^\infty 
u^{m-1}
|\partial_u^m P_{u}^\nu (x,y)|
 \, du, 
\end{equation*}
 it follows that
\begin{align*}
K_m^\nu(x,y)
& \lesssim
\int_0^\infty 
u^{m-1}
\int_0^\infty
\frac{e^{-c u^2/s}}{s^{(m+2)/2}} |W_s^\nu(x,y)| \, ds \, du \\
& \lesssim
\int_0^\infty \frac{|W_s^\nu(x,y)|}{s} \, ds.
\end{align*}


Consider firstly the situation of $\nu \geq -1/2$.
The estimate \eqref{eq:2.1} implies
\begin{align*}
& K_m^\nu(x,y)
 \lesssim
\int_0^\infty 
\frac{e^{-|x-y|^2/(4s)}}{s^{3/2}}
\frac{(1+s)^{\nu+2}e^{-\lambda_{1,\nu}^2 s}}{(s+xy)^{\nu+1/2}} \, ds \\
& \qquad  
\lesssim
\left\{
\begin{array}{ll}
\displaystyle 
 \int_0^\infty \frac{e^{-c x^2/s}}{s^{\nu+2}} \,ds
\lesssim x^{-2(\nu+1)}, & 0 < y \leq x/2 ,\\
&  \\
\displaystyle
 (xy)^{-\nu-1/2}     
\int_0^\infty \frac{e^{-c |x-y|^2/s}}{s^{3/2}} \,ds
\lesssim
\frac{(xy)^{-\nu-1/2}}{|x-y|},
&
x/2 < y \leq \min\{1,3x/2\}, \: \:x\neq y \\
&\\
\displaystyle
\int_0^\infty \frac{e^{-c y^2/s}}{s^{\nu+2}} \,ds
\lesssim y^{-2(\nu+1)}, 
&
\min\{1,3x/2\}<y \leq 1.
\end{array}
\right.
\end{align*}

Suppose now that $-1<\nu<-1/2$.
We can write
\begin{align*}
K_m^\nu(x,y)
& \lesssim
\Big(
\int_0^{xy}
+
\int_{xy}^\infty
\Big)
\frac{e^{-|x-y|^2/(4s)}}{s^{3/2}}
(s+xy)^{-\nu-1/2} \, ds \\
& =:
K_{m,1}^\nu(x,y)
+
K_{m,2}^\nu(x,y).
\end{align*}
Hence,
\begin{align*}
 K_{m,1}^\nu(x,y)
& \lesssim
\left\{
\begin{array}{ll}
\displaystyle 
 x^{-2\nu-1} \int_0^{xy} \frac{e^{-c x^2/s}}{s^{3/2}} \,ds
\lesssim x^{-2(\nu+1)}, & 0 < y \leq x/2 ,\\
&  \\
\displaystyle
 (xy)^{-\nu-1/2}     
\int_0^{xy} \frac{e^{-c |x-y|^2/s}}{s^{3/2}} \,ds
\lesssim
\frac{(xy)^{-\nu-1/2}}{|x-y|} ,
&
x/2 < y \leq \min\{1,3x/2\}, \:  \:x\neq y \\
&\\
\displaystyle
  y^{-2\nu-1}
\int_0^{xy} \frac{e^{-c y^2/s}}{s^{3/2}} \,ds
\lesssim y^{-2(\nu+1)}, 
&
\min\{1,3x/2\}<y \leq 1.
\end{array}
\right.
\end{align*}
and
\begin{align*}
 K_{m,2}^\nu(x,y)
& \lesssim
\left\{
\begin{array}{ll}
\displaystyle 
 \int_{xy}^\infty \frac{e^{-c x^2/s}}{s^{\nu + 2}} \,ds
\lesssim x^{-2(\nu+1)}, & 0 < y \leq x/2 ,\\
&  \\
\displaystyle
\int_{xy}^\infty 
\frac{e^{-c |x-y|^2/s}}{s^{\nu + 2}} \,ds
\lesssim
\frac{(xy)^{-\nu-1/2}}{|x-y|} ,
&
x/2 < y \leq  \min\{1,3x/2\}, \: \:x\neq y \\
&\\
\displaystyle
\int_{xy}^\infty \frac{e^{-c y^2/s}}{s^{\nu +2}} \,ds
\lesssim y^{-2(\nu+1)}, 
&
\min\{1,3x/2\}<y \leq 1.
\end{array}
\right.
\end{align*}
Thus, the proof of this lemma is completed.
\end{proof}

\begin{Lem}\label{Prop:2.2}
Let $\nu>-1$ and $\beta\ge 0$. We have that, for any $x,y \in (0,1)$, $ \:x\neq y,$
\begin{equation*}
\|
\partial_x H_\beta^\nu(x,y)
\|_{E_\rho} 
+
\|
\partial_y H_\beta^\nu(x,y)
\|_{E_\rho}  
\lesssim
\frac{(xy)^{-\nu-1/2}}{|x-y|^2}.
\end{equation*}
\end{Lem}

\begin{proof}
According to \eqref{A} and by proceeding as in the proof of Lemma \ref{Prop:2.1}, we get 
for $x,y\in (0,1)$, $x\neq y$,
\begin{align*}
    \|
\partial_x H_\beta^\nu(x,y)
\|_{E_\rho} &\lesssim \beta\int_0^\infty | t^{\beta-1}\partial_t^\beta\partial_x P_{t}^\nu(x,y)|dt+\int_0^\infty |t^{\beta}\partial_t^{\beta+1}\partial_x P_{t}^\nu(x,y)|dt\\
&\le  \beta\|
\partial_xH_\beta^\nu(x,y)
\|_{L^1((0,\infty),\frac{dt}{t})} +\|
\partial_xH_{\beta+1}^\nu(x,y)
\|_{L^1((0,\infty),\frac{dt}{t})}.
\end{align*}
Let  $x,y \in (0,1)$, $ \:x\neq y$, and take $m \in \mathbb{N}$ such that $m-1 \leq \gamma <m$. By arguing as in the proof of Lemma \ref{Prop:2.1} we obtain
\begin{equation*}
\|
\partial_xH_\gamma^\nu(x,y)
\|_{L^1((0,\infty), dt/t)} 
\lesssim
\int_0^\infty u^{m-1}
|\partial_u^m \partial_x P_u^\nu (x,y)| \, du.
\end{equation*}
Moreover, \cite[Lemma 3]{BCCFR} allows us to write
\begin{equation*}
|\partial_u^m \partial_x P_u^\nu (x,y)| 
\lesssim
\int_0^\infty 
\frac{e^{-cu^2/s}}{s^{(m+2)/2}}
|\partial_x W_s^\nu(x,y)| \, ds,
 \quad u>0.
\end{equation*}

Then,
\begin{equation*}
\int_0^\infty u^{m-1}
|\partial_u^m \partial_x P_u^\nu (x,y)| \, du
\lesssim
\int_0^\infty 
\frac{
|\partial_x W_s^\nu(x,y)|}{s} \, ds.
\end{equation*}

According to  \cite[Lemma 4.31]{BDL}, we have that
\begin{equation}\label{eq:1.2}
|\partial_xW_t^\nu(x,y)|
\lesssim
\frac{e^{-c (x-y)^2/t}}{(xy)^{\nu+1/2}t},
\quad \nu>-1, 
\end{equation}
which implies
\begin{align*}
\int_0^\infty u^{m-1}
|\partial_u^m \partial_x P_u^\nu (x,y)| \, du
& \lesssim
(xy)^{-\nu-1/2}
\int_0^\infty
\frac{e^{-c (x-y)^2/t}}{t^2} \, dt 
\lesssim
\frac{(xy)^{-\nu-1/2}}{|x-y|^2}, 
\end{align*}
and the proof of the property for
$\partial_xH_\gamma^\nu(x,y)$
is established.\\

By symmetry it also follows that
\begin{equation*}
\|
\partial_y H_\gamma^\nu(x,y)
\|_{L^1((0,\infty), dt/t)} 
\lesssim
\frac{(xy)^{-\nu-1/2}}{|x-y|^2}. \qedhere
\end{equation*}
\end{proof}

\begin{Rem}\label{Rem:2.1}
Littlewood-Paley-Stein $g_k$-functions associated with
$\{\phi_n^\nu\}_{n \in \NN}$ were studied by 
Ciaurri and Roncal (\cite{CR2}).
$L^p$-boundedness properties of the $g_k$-functions were established in \cite[Theorem 1]{CR2} by using vector-valued Calder\'on-Zygmund theory
(see \cite{RuFRT1,RuFRT2}). In order to do this they needed to show that the kernels of such operators satisfy the size and the regularity properties as standard Calder\'on-Zygmund kernels. These facts are proved in \cite[Proposition 2 and 3]{CR2}. The method implemented uses as a fundamental tool the following asymptotic representation for the Bessel function $J_\nu$ (\cite[p. 122]{Leb}). For every $M \in \NN$ there exist real numbers $A_{\nu,j}$ and $B_{\nu,j}$,
$j=0,\cdots, M$ such that
\begin{equation*}
\Big| 
\sqrt{z} J_\nu(z)
-
\sum_{j=0}^M
\Big(
\frac{A_{\nu,j}}{z^j} \sin z
+
\frac{B_{\nu,j}}{z^j} \cos z
\Big)
\Big|
\lesssim
|z|^{-(M+1)}, \quad z>1.
\end{equation*}
The proofs of \cite[Propositions 2 and 3]{CR2}
need subtle and large manipulations which cover up almost the entire paper.
\cite[Propositions 2 and 3]{CR2} have certain resemblance to our Lemmas \ref{Prop:2.1} and \ref{Prop:2.2}.
However, by using \eqref{eq:2.1} and \eqref{eq:1.2} 
 it is possible to deduce \cite[Propositions 2 and 3]{CR2} in a faster way.
\end{Rem}
Note that according to Lemma \ref{Prop:2.1}, for every $f \in L^p((0,1),x^{2\nu+1}dx)$ with $1\leq p \leq \infty$,
\begin{equation*}
\int_0^1 \|H_\beta^\nu(x,y)\|_{\rho} |f(y)| y^{2\nu+1} \, dy    
< \infty, \quad x \notin \supp f.
\end{equation*}

 Let {$f \in L^\infty(0,1)$.}
We are going to see that
\begin{equation}\label{eq:zerosupp}
0
=
\|\mathbb{T}_\beta^\nu(f)(x) - T_\beta^\nu(f)(\cdot,x) \|_{\rho}, \quad x \notin \supp f.
\end{equation}

Indeed, let $0<a<b{<\infty}$. We define $L_{a,b}$ as follows
\begin{equation*}
L_{a,b}(\Psi)
:=
\Psi(b)-\Psi(a), \quad \Psi \in E_\rho.
\end{equation*}
We have that 
\begin{equation*}
|L_{a,b}(\Psi)|
\leq 
\|\Psi\|_{\rho}, \quad \Psi \in E_\rho,
\end{equation*}
that is, $L_{a,b} \in (E_\rho)'$, the dual space of $E_\rho$. Then,
\begin{align*}
L_{a,b}(\mathbb{T}_\beta^\nu(f)(x))
& = 
\int_0^1 L_{a,b} \Big(H_\beta^\nu(x,y)\Big) f(y) y^{2\nu+1} \, dy \\
& = 
\int_0^1 
\Big(
t^\beta \partial_t^\beta P_t^\nu(x,y)_{|_{t=b}}
-
t^\beta \partial_t^\beta P_t^\nu(x,y)_{|_{t=a}}
\Big)
f(y) y^{2\nu+1} \, dy \\
& = T_\beta^\nu(f)(b,x) - T_\beta^\nu(f)(a,x), \quad x \notin \supp f.
\end{align*}
Hence, \eqref{eq:zerosupp} follows.\\

We have all the ingredients to prove the $L^p$-boundedness properties for the operator 
$\mathcal{V}_\rho(\{t^\beta \partial_t^\beta P_t^\nu\}_{t>0})$.\\

According to \cite[Lemma 9]{CR2}, by denoting $m_\nu$ to the measure on $(0,1)$ with density function $x^{2\nu+1}$ with respect to the Lebesgue measure, we have that
\begin{equation*}
m_\nu(B(x,|x-y|))
\lesssim
\left\{
\begin{array}{ll}
x^{2\nu+2},  
& 0 < y \leq x/2,  \\
& \\
(xy)^{\nu+1/2}|x-y|, 
& x/2 < y \leq \min\{1,3x/2\}, \\
& \\
y^{2\nu+2}, & \min\{1,3x/2\} < y \leq 1.
\end{array}
\right.
\end{equation*}
By using Lemmas \ref{Prop:2.1} and \ref{Prop:2.2}
we obtain
\begin{equation*}
\|t^\beta \partial_t^\beta P_t^\nu(x,y)\|_{\rho}  
\lesssim
\frac{1}{m_\nu(B(x,|x-y|))}, \quad x,y \in (0,1),
\end{equation*}
and for 
$0<x/2<y<\min\{1,3x/2\}$
\begin{equation*}
\|\partial_x[t^\beta \partial_t^\beta P_t^\nu(x,y)]\|_{\rho}  
+
\|\partial_y[t^\beta \partial_t^\beta P_t^\nu(x,y)]\|_{\rho}  
\lesssim 
\frac{1}{|x-y|m_\nu(B(x,|x-y|))}.
\end{equation*}
Then, by applying the local Calder\'on-Zygmund theorem in the measure space $((0,1),m_\nu)$ we conclude that the local operator $\mathbb{T}_{\beta,loc}^\nu$ defined by
\begin{equation*}
\mathbb{T}_{\beta,loc}^\nu  
(f)(x)
:=
\int_{x/2}^{\min\{1,3x/2\}}
H_\beta^\nu(x,y) f(y) y^{2\nu+1} dy, \quad x \in (0,1),
\end{equation*}
is bounded from 
$L^p((0,1),x^{2\nu+1}dx)$ into 
$L^p_{E_\rho}((0,1),x^{2\nu+1}dx)$, for every 
$1<p<\infty$, and from 
$L^1((0,1),x^{2\nu+1}dx)$ into 
$L^{1,\infty}_{E_\rho}((0,1),x^{2\nu+1}dx)$.
Also, by defining
\begin{equation*}
T_{\beta,loc}^\nu  
(f)(t,x)
:=
\int_{x/2}^{\min\{1,3x/2\}}
t^\beta \partial_t^\beta P_t^\nu(x,y) f(y) y^{2\nu+1} dy, \quad x \in (0,1), \quad t>0,
\end{equation*}
and since
\begin{equation*}
\|T_{\beta,loc}^\nu(f)(\cdot,x)\|_\rho    
=
\|\mathbb{T}_{\beta,loc}^\nu(f)(x)\|_\rho,
\quad x \in (0,1),
\end{equation*}
we conclude that 
$T_{\beta,loc}^\nu$ is bounded from 
$L^p((0,1),x^{2\nu+1}dx)$ into 
$L^p_{E_\rho}((0,1),x^{2\nu+1}dx)$, for every 
$1<p<\infty$, and from 
$L^1((0,1),x^{2\nu+1}dx)$ into 
$L^{1,\infty}_{E_\rho}((0,1),x^{2\nu+1}dx)$.\\

We now consider the following Hardy-type operators
\begin{equation*}
H_0^\nu(g)(x)
:=
\frac{1}{x^{2\nu+2}}
\int_0^x y^{2\nu+1} g(y) \, dy, \quad x \in (0,\infty),
\end{equation*}
and
\begin{equation*}
H_\infty(g)(x)
:=
\int_x^\infty \frac{g(y)}{y} \, dy, \quad x \in (0,\infty).
\end{equation*}
According to \cite[Lemmas 1 and 2]{BHNV} the operators $H_0^\nu$ and $H_\infty$ are bounded from 
$L^p((0,1),x^{2\nu+1}dx)$ into itself, for every $1<p<\infty$, and from $L^1((0,1),x^{2\nu+1}dx)$ into $L^{1,\infty}((0,1),x^{2\nu+1}dx)$.\\

We define the operator
\begin{equation*}
T_{\beta,glob}^\nu  
(f)(t,x)
:=
\int_{(0,1)\setminus(x/2,\min\{1,3x/2\})}
t^\beta \partial_t^\beta P_t^\nu(x,y) f(y) y^{2\nu+1} dy, \quad x \in (0,1), \quad t>0,
\end{equation*}
By using Lemma \ref{Prop:2.1} 
and the $L^p$-boundedness properties of the Hardy-type operators $H_0^\nu$ and $H^\infty$ we deduce that $T_{\beta,glob}^\nu$ is bounded from 
$L^p((0,1),x^{2\nu+1}dx)$ into 
$L^p_{E_\rho}((0,1),x^{2\nu+1}dx)$,
$1<p<\infty$, and from 
$L^1((0,1),x^{2\nu+1}dx)$ into 
$L^{1,\infty}_{E_\rho}((0,1),x^{2\nu+1}dx)$.\\

Furthermore, since 
$
T_{\beta}^\nu
=
T_{\beta,loc}^\nu
+
T_{\beta,glob}^\nu$
we conclude that 
$T_{\beta}^\nu$
is bounded from 
$L^p((0,1),$ $x^{2\nu+1}dx)$ into 
$L^p_{E_\rho}((0,1),x^{2\nu+1}dx)$,
$1<p<\infty$, and from 
$L^1((0,1),x^{2\nu+1}dx)$ into 
$L^{1,\infty}_{E_\rho}((0,1),x^{2\nu+1}dx)$.
This finishes the proof of Theorem \ref{Th:1.1}, $i)$.

\subsection{The $\lambda$-jump operator
$\lambda [\Lambda(\{t^\beta \partial_t^\beta \mathcal{P}_t^\nu\}_{t>0},\lambda)]^{1/p}$}
\quad \\

On the other hand, recalling that
(see \cite[p. 6712]{JSW})
\begin{equation*}
\lambda [\Lambda(\{t^\beta \partial_t^\beta P_t^\nu\}_{t>0},\lambda)]^{1/p}
\lesssim
\mathcal{V}_\rho(\{t^\beta \partial_t^\beta P_t^\nu\}_{t>0})(f),
\quad \lambda>0,
\end{equation*}
the $L^p$-boundedness properties of the $\lambda$-jump operator
$\lambda [\Lambda(\{t^\beta \partial_t^\beta P_t^\nu\}_{t>0},\lambda)]^{1/p}$
can be deduced from the corresponding ones for the variation operator
$\mathcal{V}_\rho(\{t^\beta \partial_t^\beta P_t^\nu\}_{t>0})$.\\

\subsection{The oscillation operator 
$\mathcal{O}(\{t^\beta \partial_t^\beta P_t^\nu\}_{t>0},$ $\{t_j\}_{j \in \NN})$}
\quad \\

We now establish the $L^p$-boundedness properties for the oscillation operator 
$\mathcal{O}(\{t^\beta \partial_t^\beta P_t^\nu\}_{t>0},$ $\{t_j\}_{j \in \NN})$, where $\{t_j\}_{j \in \NN}$ is a decreasing sequence in $(0,\infty)$.\\

We can write, for $x, y \in (0,1)$ and $j \in \NN$,
\begin{align*}
\sup_{t_{j+1} 
\leq \varepsilon_{j+1}
< \varepsilon_j
\leq t_j}    
\Big|
t^\beta \partial_t^\beta P_t^\nu(x,y)_{|_{t=\varepsilon_j}}
-
t^\beta \partial_t^\beta P_t^\nu(x,y)_{|_{t=\varepsilon_{j+1}}}
\Big|
\leq
\int_{t_{j+1}}^{t_j}
|\partial_t[t^\beta \partial_t^\beta P_t^\nu(x,y)]| \, dt.
\end{align*}
Then, it follows that
\begin{equation*}
\mathcal{O}(\{t^\beta \partial_t^\beta P_t^\nu\}_{t>0},\{t_j\}_{j \in \NN})
\leq
\int_0^\infty
|\partial_t[t^\beta \partial_t^\beta P_t^\nu(x,y)]| \, dt, \quad
x, y \in (0,1).
\end{equation*}
We also have, for every $x, y \in (0,1)$,
\begin{equation*}
\mathcal{O}(\{t^\beta \partial_t^\beta \partial_xP_t^\nu\}_{t>0},\{t_j\}_{j \in \NN})
\leq
\int_0^\infty
|\partial_t[t^\beta \partial_t^\beta \partial_xP_t^\nu(x,y)]| \, dt,
\end{equation*}
and
\begin{equation*}
\mathcal{O}(\{t^\beta \partial_t^\beta \partial_yP_t^\nu\}_{t>0},\{t_j\}_{j \in \NN})
\leq
\int_0^\infty
|\partial_t[t^\beta \partial_t^\beta \partial_yP_t^\nu(x,y)]| \, dt.
\end{equation*}
Therefore, by using \eqref{eq:I3}
and Minkowski inequality we obtain,
for $x\in (0,1)$,
\begin{align}\label{eq:I4}
\mathcal{O}(\{t^\beta \partial_t^\beta P_t^\nu\}_{t>0},\{t_j\}_{j \in \NN})(f)(x)
&\lesssim
\sum_{k_1+2k_2=m+1}
\int_0^\infty
\frac{z^{m-\beta-1}}{(1+z)^m}
\int_0^\infty e^{-v} v^{k_1+k_2-3/2} \nonumber \\
& \qquad \times
\mathcal{O}\Big(\{W_t^\nu\}_{t>0},
\Big\{
t_j^2 \frac{(1+z)^2}{4v}\Big\}_{j \in \NN}\Big)(f)(x)
\, dv \, dz. 
\end{align}
According to \cite[Theorem 3.3]{JR}, for every decreasing sequence $\{s_j\}_{j \in \NN}$ in $(0,\infty)$ and  $f \in L^2((0,1),x^{2\nu+1}(\phi_1^\nu(x))^2dx)$, we have that
\begin{align*}
\|
\mathcal{O}(\{\WW_t^\nu\}_{t>0},\{s_j\}_{j \in \NN})(f)
\|_{L^2((0,1),x^{2\nu+1}(\phi_1^\nu(x))^2dx)} 
\lesssim
\|f\|_{L^2((0,1),x^{2\nu+1}(\phi_1^\nu(x))^2dx)},
\end{align*}
 where the constant does not depend on 
$\{s_j\}_{j \in \NN}$.\\

Let $\{s_j\}_{j \in \NN}$ be a decreasing sequence in $(0,\infty)$.
By proceeding as in the variation operator case we get,
for $x \in (0,1)$,
\begin{align*}
\mathcal{O}(\{W_t^\nu\}_{t>0},\{s_j\}_{j \in \NN})(f)(x)
& 
\lesssim
\mathcal{O}(\{\widetilde{W}_t^\nu\}_{t>0},\{s_j\}_{j \in \NN})(f)(x)
+
\widetilde{W}_*^\nu(f)(x) \\
& 
\lesssim
\phi_1^\nu(x)
\mathcal{O}(\{\mathbb{W}_t^\nu\}_{t>0},\{s_j\}_{j \in \NN})
( f/\phi_1^\nu)(x)
+
\widetilde{W}_*^\nu(f)(x),
\end{align*}
where the constant does not depend on 
$\{s_j\}_{j \in \NN}$.
Hence, there exists $C>0$ independent of 
$\{s_j\}_{j \in \NN}$ such that
\begin{align*}
\|
\mathcal{O}(\{W_t^\nu\}_{t>0},\{s_j\}_{j \in \NN})(f)
\|_{L^2((0,1),x^{2\nu+1}dx)} 
\leq 
C
\|f\|_{L^2((0,1),x^{2\nu+1}dx)},
\end{align*}
for $f \in L^2((0,1),x^{2\nu+1}dx)$.\\

Minkowski inequality and \eqref{eq:I4} lead to
\begin{align*}
& \|
\mathcal{O}(\{t^\beta \partial_t^\beta P_t^\nu\}_{t>0},\{t_j\}_{j \in \NN})(f)
\|_{L^2((0,1),x^{2\nu+1}dx)}    \\
& \qquad \lesssim
\sum_{k_1+2k_2=m+1}
\int_0^\infty
\frac{z^{m-\beta-1}}{(1+z)^m}
\int_0^\infty e^{-v} v^{k_1+k_2-3/2} \nonumber \\
& \hspace{4cm}  \times
\Big\|\mathcal{O}\Big(\{W_t^\nu\}_{t>0},
\Big\{
t_j^2 \frac{(1+z)^2}{4v}\Big\}_{j \in \NN}\Big)(f)\Big\|_{L^2((0,1),x^{2\nu+1}dx)}
\, dv \, dz \\
& \qquad \lesssim
\|f\|_{L^2((0,1),x^{2\nu+1}dx)}, \quad
f \in L^2((0,1),x^{2\nu+1}dx).
\end{align*}

We now consider the space $E$ that consists of all those functions $\Psi : (0,\infty) \longrightarrow \CC$ such that
\begin{equation*}
\|\Psi\|_{E}
:=
\Big( 
\sum_{j=1}^\infty
\sup_{t_{j+1} 
\leq \varepsilon_{j+1}
< \varepsilon_j
\leq t_j}   
|\Psi(\varepsilon_j)
-
\Psi(\varepsilon_{j+1})|^2
\Big)^{1/2}
< \infty.
\end{equation*}
Proceeding similarly to the analysis of
$\mathcal{V}_\rho(\{t^\beta \partial_t^\beta P_t^\nu\}_{t>0})$, by replacing the space $E_\rho$ by $E$ we can show the $L^p$-boundedness properties for the oscillation operator 
$\mathcal{O}(\{t^\beta \partial_t^\beta P_t^\nu\}_{t>0},\{t_j\}_{j \in \NN})$. \\

\subsection{The short variation operator
$\mathcal{S}_V(\{t^\beta \partial_t^\beta P_t^\nu\}_{t>0})$}
\quad \\

We now study the short variation operator
$\mathcal{S}_V(\{t^\beta \partial_t^\beta P_t^\nu\}_{t>0})$. Let $k \in \ZZ$.
Observe that for $x \in (0,1)$,
\begin{align*}
& \mathcal{V}_k(\{t^\beta \partial_t^\beta P_t^\nu\}_{t>0})(f)(x) \\
& \qquad 
=
\sup_{
\substack{2^{-k} < \varepsilon_\ell < \varepsilon_{\ell-1} < \dots < \varepsilon_1 <2^{-k+1} \\
\ell \in \NN}}
\Big(
\sum_{j=1}^{\ell-1}
\Big| 
t^\beta \partial_t^\beta P_t^\nu(f)(x)_{|_{t=\varepsilon_j}}
-
t^\beta \partial_t^\beta P_t^\nu(f)(x)_{|_{t=\varepsilon_{j+1}}}
\Big|^2
\Big)^{1/2} \\
& \qquad \leq 
\sup_{
\substack{2^{-k} < \varepsilon_\ell < \varepsilon_{\ell-1} < \dots < \varepsilon_1 <2^{-k+1} \\
\ell \in \NN}}
\sum_{j=1}^{\ell-1}
\Big| 
t^\beta \partial_t^\beta P_t^\nu(f)(x)_{|_{t=\varepsilon_j}}
-
t^\beta \partial_t^\beta P_t^\nu(f)(x)_{|_{t=\varepsilon_{j+1}}}
\Big| \\
&  \qquad \leq 
\sup_{
\substack{2^{-k} < \varepsilon_\ell < \varepsilon_{\ell-1} < \dots < \varepsilon_1 <2^{-k+1} \\
\ell \in \NN}}
\sum_{j=1}^{\ell-1}
\int_{\varepsilon_{j+1}}^{\varepsilon_{j}}
| \partial_t[
t^\beta \partial_t^\beta P_t^\nu(f)(x)]
| \, dt\\
&  \qquad \leq 
\int_{2^{-k}}^{2^{-k+1}}
| 
t^{\beta-1} \partial_t^\beta P_t^\nu(f)(x)
| \, dt
+
\int_{2^{-k}}^{2^{-k+1}}
| 
t^{\beta} \partial_t^{\beta+1} P_t^\nu(f)(x)
| \, dt \\
&  \qquad \leq 
\Big( 
\int_{2^{-k}}^{2^{-k+1}}
\frac{dt}{t}\Big)^{1/2}
\Big[
\Big(
\int_{2^{-k}}^{2^{-k+1}}
| 
t^{\beta} \partial_t^\beta P_t^\nu(f)(x)
|^2 \frac{dt}{t} \Big)^{1/2} \\
& \hspace{4cm} +
\Big(
\int_{2^{-k}}^{2^{-k+1}}
| 
t^{\beta+1} \partial_t^{\beta+1} P_t^\nu(f)(x)
|^2 \frac{dt}{t} \Big)^{1/2}
\Big] \\
&  \qquad \leq 
C
\Big[
\Big(
\int_{2^{-k}}^{2^{-k+1}}
| 
t^{\beta} \partial_t^\beta P_t^\nu(f)(x)
|^2 \frac{dt}{t} \Big)^{1/2}  +
\Big(
\int_{2^{-k}}^{2^{-k+1}}
| 
t^{\beta+1} \partial_t^{\beta+1} P_t^\nu(f)(x)
|^2 \frac{dt}{t} \Big)^{1/2}
\Big].
\end{align*}
Here $C>0$ does not depend on $k$.\\

Then,
\begin{align}\label{eq:I5}
\mathcal{S}_V(\{t^\beta \partial_t^\beta P_t^\nu\}_{t>0})
& =
\Big( 
\sum_{k=-\infty}^\infty
|
\mathcal{V}_k(\{t^\beta \partial_t^\beta P_t^\nu\}_{t>0})(f)(x)
|^2
\Big)^{1/2} \nonumber \\
& \lesssim
g_{\beta}^\nu(f)(x)
+
g_{\beta+1}^\nu(f)(x), \quad x \in (0,1),
\end{align}
where, for every $\gamma>0$, the
$\gamma$-Littlewood-Paley-Stein function
$g_\gamma^\nu$ is defined by
\begin{equation*}
g_\gamma^\nu(f)(x)
:= 
\Big(
\int_{0}^{\infty}
| 
t^{\gamma} \partial_t^\gamma P_t^\nu(f)(x)
|^2 \frac{dt}{t} \Big)^{1/2},
\quad x \in (0,1).
\end{equation*}
In \cite{CR2} Littlewood-Paley-Stein functions associated with Fourier-Bessel expansions were considered. They defined, for every $k \in \NN$,
\begin{equation*}
G_k^\nu(f)(x)
:= 
\Big(
\int_{0}^{\infty}
| 
t^{k} \partial_t^k \widetilde{P}_t^\nu(f)(x)
|^2 \frac{dt}{t} \Big)^{1/2},
\quad x \in (0,1),
\end{equation*}
where 
\begin{equation*}
\widetilde{P}_t^\nu(f)
:=
\sum_{n=1}^\infty e^{-tn} c_n^\nu(f)\phi_n^\nu,
\quad t>0.
\end{equation*}
Note that the Poisson semigroups 
$\{\widetilde{P}_t^\nu\}_{t>0}$
and 
$\{P_t^\nu\}_{t>0}$
are not the same.\\

We now establish the $L^p$-boundedness properties for $g_\gamma^\nu$, $\gamma>0$.
Our result can be seen as an extension of \cite[Theorem 1]{CR2}. We will prove it by using again \cite[Theorem 1]{MSZ}.
Our procedure is different than the one used in \cite{CR2} (see Remark \ref{Rem:2.1}).

\begin{Prop}\label{Prop:2.7}
Let $\gamma>0$ and $\nu>-1$. Then, the operator $g_\gamma^\nu$ is  
bounded from 
$L^p((0,1),x^{2 \nu + 1}dx)$ into itself, for
every $1<p<\infty$, and from $L^1((0,1),$ $x^{2 \nu + 1}dx)$ into 
$L^{1,\infty}((0,1),x^{2 \nu + 1}dx)$.
\end{Prop}

\begin{proof}
Since (see 
\cite[p. 111]{BS1}
and
\cite[pp. 104 and 122]{Leb})
\begin{equation*}
d_{n,\nu}
=
\sqrt{\pi}
\Big( 1 + O\Big(\frac{1}{n}\Big)  \Big),
\quad n \in \NN,
\end{equation*}
\begin{equation*}
J_\nu(z)
\sim
\frac{z^\nu}{2^{\nu} \Gamma(\nu+1)}, \quad
z \to 0^+
\end{equation*}
and
\begin{equation*}
J_\nu(z)
=
O\Big( \frac{1}{\sqrt{z}}\Big), \quad
z \to  + \infty,
\end{equation*}
we get
\begin{equation*}
|\phi_n^\nu(x)|
\lesssim
\Big( 
\lambda_{n,\nu}^{\nu+1/2}
+
x^{-\nu-1/2}
\Big),
\quad n \in \NN, \quad x \in (0,1).
\end{equation*}

{ Recall also (see \cite[p. 618]{Wa}) that
 \begin{equation}\label{eq:lambdannu}
 \lambda_{n,\nu}  =
 \pi \Big( n+\frac{\nu}{2}-\frac{1}{4}\Big)
 +
 O\Big( \frac{1}{n}\Big),
 \quad n \in \mathbb{N}.
 \end{equation}}
Then, the differentiation under the summation sign is justified and we can obtain
\begin{equation*}
\partial_t^k P_t^\nu(f)(x)
= 
\sum_{n=1}^\infty 
e^{- \lambda_{n,\nu}t} 
(- \lambda_{n,\nu})^k
c_n^\nu(f)
\phi_n^\nu(x), \quad
x \in (0,1), \quad t>0,
\end{equation*}
for every 
$f \in L^2((0,1),x^{2 \nu + 1}dx)$ and $k \in \NN$.\\

Let $f \in L^2((0,1),x^{2 \nu + 1}dx)$ .
We choose $m \in \NN$ such that 
$m-1\leq \gamma <m$ and write
\begin{align*}
\partial_t^\gamma P_t^\nu(f)(x)    
& = 
\frac{-1}{\Gamma(m-\gamma)}
\int_0^\infty 
\partial_t^m 
P_{t+s}^\nu(f)(x) s^{m-\gamma-1} \, ds \\
& = 
\frac{-1}{\Gamma(m-\gamma)}
\sum_{n=1}^\infty 
(- \lambda_{n,\nu})^m
c_n^\nu(f)
\phi_n^\nu(x)
\int_0^\infty 
e^{- \lambda_{n,\nu}(t+s)} 
 s^{m-\gamma-1} \, ds \\
& = (-1)^{m+1}
\sum_{n=1}^\infty 
(\lambda_{n,\nu})^\gamma
c_n^\nu(f)
\phi_n^\nu(x)
e^{- \lambda_{n,\nu}t}, \quad
x \in (0,1), \quad t>0.
\end{align*}
We have that
\begin{align*}
\|g_\gamma^\nu(f)\|_{L^2((0,1),x^{2 \nu + 1}dx)}^2 
& =
\int_0^\infty 
\int_0^\infty 
|t^\gamma \partial_t^\gamma P_t^\nu(f)(x) |^2 
\frac{dt}{t}
x^{2\nu+1} \, dx \\
& =
\int_0^\infty 
\int_0^\infty 
\Big|
\sum_{n=1}^\infty 
(t\lambda_{n,\nu})^\gamma
c_n^\nu(f)
e^{- \lambda_{n,\nu}t}
\phi_n^\nu(x)
\Big|^2 
x^{2\nu+1} \, dx
\frac{dt}{t}\\
& =
\sum_{n=1}^\infty 
|c_n^\nu(f)|^2
\int_0^\infty
(t\lambda_{n,\nu})^{2\gamma}
e^{- 2\lambda_{n,\nu}t}
\frac{dt}{t}\\
& =
\frac{\Gamma(2\gamma)}{2^{2\gamma}}
\sum_{n=1}^\infty 
|c_n^\nu(f)|^2 
=
\frac{\Gamma(2\gamma)}{2^{2\gamma}}
\|f\|_{L^2((0,1),x^{2 \nu + 1}dx)}^2.
\end{align*}
In order to prove our result we use vector-valued Calder\'on-Zygmund theory (\cite{RuFRT1,RuFRT2}).\\

We consider the operator $\mathcal{G}_\gamma^\nu$ defined by
\begin{equation*}
\mathcal{G}_\gamma^\nu(f)(t,x)
:= 
t^\gamma \partial_t^\gamma P_t^\nu(f)(x),
\quad x \in (0,1), \quad t>0.
\end{equation*}
We have just proved that $\mathcal{G}_\gamma^\nu$ is bounded from $L^2((0,1),x^{2 \nu + 1}dx)$ into
$L^2_{L^2((0,\infty),dt/t)}((0,1),$ $x^{2 \nu + 1}dx)$.\\

Define
\begin{equation*}
K_\gamma^\nu(t,x,y)
:=
t^\gamma \partial_t^\gamma P_t^\nu(x,y),
\quad x,y \in (0,1), \quad t>0.
\end{equation*}
We are going to see that, for any $x,y \in (0,1)$,
\begin{equation}\label{eq:I6}
\|K_\gamma^\nu(\cdot,x,y)\|_{L^2((0,\infty),dt/t)}
\lesssim
\left\{
\begin{array}{ll}
x^{-2(\nu+1)},  
& 0 < y \leq x/2,  \\
& \\
(xy)^{-\nu-1/2}|x-y|^{-1}, 
& x/2 < y \leq \min\{1,3x/2\},  \:x\neq y, \\
& \\
y^{-2(\nu+1)}, & 
\min\{1,3x/2\} < y \leq 1.
\end{array}
\right.
\end{equation}
We choose $m \in \NN$ such that $m-1 \leq \gamma <m$.
According to \eqref{eq:2.1.1} we can write
\begin{equation*}
\partial_t^m P_t^\nu(x,y)
=
\frac{1}{2\sqrt{\pi}}
\int_0^\infty 
\frac{\partial_t^m[te^{-t^2/(4u)}]}{u^{3/2}}
W_u^\nu(x,y) \, du, \quad 
x,y \in (0,1), \quad t>0.
\end{equation*}
Then, for $x,y \in (0,1)$ and $t>0$,
\begin{equation*}
\partial_t^\gamma P_t^\nu(x,y)
=
\frac{-1}{2\sqrt{\pi}\Gamma(m-\beta)}
\int_0^\infty 
\int_0^\infty 
\frac{\partial_v^m[ve^{-v^2/(4u)}]_{|_{v=t+s}}}{u^{3/2}}
W_u^\nu(x,y) \, du
s^{m-\gamma-1} \, ds.
\end{equation*}
By \cite[Lemma 3]{BCCFR} we get
\begin{align*}
|\partial_t^\gamma P_t^\nu(x,y)|  
& \lesssim
\int_0^\infty 
\int_0^\infty 
\frac{|\partial_v^m[ve^{-v^2/(4u)}]|_{|_{v=t+s}}}{u^{3/2}}
|W_u^\nu(x,y)| \, du
s^{m-\gamma-1} \, ds \\
& \lesssim
\int_0^\infty 
|W_u^\nu(x,y)|
\int_0^\infty 
\frac{e^{-(t+s)^2/(8u)}}{u^{(m+2)/2}}
s^{m-\gamma-1} \, ds \, du \\
& \lesssim
\int_0^\infty 
|W_u^\nu(x,y)|
\frac{e^{-ct^2/u}}{u^{(m+2)/2}}
\int_0^\infty 
e^{-cs^2/u} 
s^{m-\gamma-1} \, ds \, du \\
& \lesssim
\int_0^\infty 
|W_u^\nu(x,y)|
\frac{e^{-ct^2/u}}{u^{(\gamma+2)/2}}
\, du,
\quad x,y \in (0,1), \quad t>0.
\end{align*}
Minkowski inequality leads to
\begin{align*}
\|K_\gamma^\nu(\cdot,x,y)\|_{L^2((0,\infty),dt/t)} 
& \lesssim
\int_0^\infty 
\frac{|W_u^\nu(x,y)|}{u^{(\gamma+2)/2}}
\Big(
\int_0^\infty
t^{2\gamma-1}
e^{-ct^2/u}
\Big)^{1/2}
\, du \\
& \lesssim
\int_0^\infty 
\frac{|W_u^\nu(x,y)|}{u}
\, du,
\quad x,y \in (0,1).
\end{align*}
As in Lemma \ref{Prop:2.1} we can now establish \eqref{eq:I6}.\\

Next, by proceeding as in Lemma \ref{Prop:2.2} we obtain,
for $x,y \in (0,1)$, $x\neq y$,
\begin{align*}
\|\partial_x K_\gamma^\nu(\cdot,x,y)\|_{L^2((0,\infty),dt/t)}     
+
\|\partial_y K_\gamma^\nu(\cdot,x,y)\|_{L^2((0,\infty),dt/t)} 
\lesssim
\frac{(xy)^{-\nu-1/2}}{|x-y|^2}.
\end{align*}
We define the operator
\begin{equation*}
T_\gamma^\nu(f)(x)
:=
\int_0^1 \mathbb{K}_\gamma^\nu(x,y)
f(y) y^{2\nu+1} \, dy, \quad x \in (0,1),
\end{equation*}
where for every $x,y \in (0,1)$, $x \neq y$,
\begin{equation*}
\begin{array}{rrccl}
\mathbb{K}_\gamma^\nu(x,y) & : & (0,\infty) & \longrightarrow & \mathbb{R}   \\
&& t & \longmapsto &
K_\gamma^\nu(t,x,y),
\end{array}
\end{equation*}
and the integral is understood in the 
$L^2((0,\infty),dt/t)$-B\"ochner sense.\\

Note that if
$f \in 
L^\infty((0,1),x^{2\nu+1}dx)
=
L^\infty((0,1),dx)
$,
from \eqref{eq:I6} it follows that
\begin{equation*}
\int_{0}^1 
\|K_\gamma^\nu(x,y)\|_{L^2((0,\infty),dt/t)} \, |f(y)|
\, dy    
\lesssim
\int_0^1 \frac{|f(y)|}{|x-y|} \, dy
< \infty, 
\quad x \in (0,1) \setminus \supp f.
\end{equation*}
Assume that 
$f \in 
L^\infty((0,1),dx)
$
with compact support on $(0,1)$ and
$g \in L^2((0,\infty),$ $dt/t)$. By using again \eqref{eq:I6} we get
\begin{align*}
& \int_0^\infty
\int_0^1 |K_\gamma^\nu(t,x,y)| \,
|f(y)| \, |g(t)| y^{2\nu+1}\, dy \, \frac{dt}{t} \\
& \qquad \leq 
\|g\|_{L^2((0,\infty),dt/t)}
\int_0^1 
\|K_\gamma^\nu(\cdot,x,y)\|_{L^2((0,\infty),dt/t)} \,
|f(y)| y^{2\nu+1} \, dy \\
& \qquad \lesssim
\|g\|_{L^2((0,\infty),dt/t)}
\int_{\supp f}
{|f(y)|} 
\, dy 
< \infty,
\quad x \in (0,1) \setminus \supp f.
\end{align*}
Then, by taking into account properties of the B\"ochner integral we can write
\begin{align*}
& \int_0^\infty    
g(t) [T_\gamma^\nu(f)(x)](t) \, dt
 =
\int_0^1 \int_0^\infty 
K_\gamma^\nu(t,x,y)
g(t)
\frac{dt}{t}
f(y)
y^{2\nu+1}
\, dy\\
& \qquad = 
\int_0^\infty 
g(t)
\int_0^1 
K_\gamma^\nu(t,x,y)
f(y)
y^{2\nu+1}
\, dy
\frac{dt}{t} ,
\quad x \in (0,1) \setminus \supp f.
\end{align*}
Hence,
\begin{equation*}
[T_\gamma^\nu(f)(x)](t)
=
t^\gamma \partial_t^\gamma P_t^\nu(f)(x),
\quad x \in (0,1) \setminus \supp f,
\end{equation*}
in the sense of equality in 
$L^2((0,\infty),dt/t)$.\\

By using now vector-valued 
Calder\'on-Zygmund theory and 
$L^p$-boundedness properties of the Hardy type operators $H_0^\nu$ and $H_\infty$ we can finish the proof of this proposition by proceeding as in the variation operator  
$\mathcal{V}_\rho$ case.
\end{proof}

\begin{Rem}\label{Rem:2.6}
Since $\{W_t^\nu\}_{t>0}$ is not a symmetric diffusion semigroup, 
Proposition \ref{Prop:2.7}
cannot be seen as a special case of 
\cite[Theorems 1.1 and 1.2]{TZ}.
\end{Rem}

\begin{proof}[Proof of Theorem \ref{Th:1.1}, $iv)$]
From \eqref{eq:I5} and 
Proposition \ref{Prop:2.7} we can deduce that the short variation operator 
$\mathcal{S}_V(\{t^\beta \partial_t^\beta P_t^\nu\}_{t>0})$
is bounded from 
$L^p((0,1),x^{2 \nu + 1}dx)$ into itself, for
every $1<p<\infty$, and from $L^1((0,1),x^{2 \nu + 1}dx)$ into 
$L^{1,\infty}((0,1),x^{2 \nu + 1}dx)$.
\end{proof}

\section{Proof of Theorem \ref{Th:1.3}}
\label{Sect3}

For every $n\in\NN$, we have that $\Psi_n^\nu(x)=x^{\nu+1/2}\phi_n^\nu(x),$ $x\in (0,1).$ Then,
$$
\mathcal{P}_t^\nu(f)(x)=\int_0^1\mathcal{P}_t^\nu(x,y)f(y)dy, \:\; x\in (0,1), \text{ and } t>0,
$$
where 
$$\mathcal{P}_t^\nu(x,y)
:=(xy)^{\nu+1/2}P_t^\nu(x,y), \quad 
x,y\in (0,1), \quad t>0.$$

We consider the operator $\mathcal{T}_\beta^\nu$ defined by
\begin{equation*}
\mathcal{T}_\beta^\nu(f)(t,x)
:=
t^\beta \partial_t^\beta \mathcal{P}_t^\nu(f)(x),
\quad x \in (0,1), \quad t>0.
\end{equation*}
Also, we define
\begin{equation*}
{\bf T}_\beta^\nu(f)(x)
:=
\int_0^1 \mathcal{H}_\beta^\nu(x,y) f(y)  dy,
\end{equation*}
where, for every $x,y \in (0,1)$, $x \neq y$,
\begin{equation*}
\begin{array}{rrccl}
\mathcal{H}_\beta^\nu(x,y) & : & (0,\infty) & \longrightarrow & \mathbb{R}   \\
&& t & \longmapsto &
[\mathcal{H}_\beta^\nu(x,y)](t)
:= t^\beta \partial_t^\beta \mathcal{P}_t^\nu(x,y),
\end{array}
\end{equation*}
and the integral is understood in the 
$E_\rho/\sim$ B\"ochner sense.\\

We want  to see that the operator $\mathcal{T}_\beta^\nu$ is a $E_\rho$-Calder\'on-Zygmund operator associated with the kernel $\mathcal{H}_\beta^\nu (x,y).$
 Observe that $\mathcal{H}_\beta^\nu(x,y)=(xy)^{\nu+1/2}{H}_\beta^\nu(x,y)$. Then, by using Lemmas \ref{Prop:2.1} and \ref{Prop:2.2}, we obtain that, for $x,y\in (0,1),$
\begin{equation}\label{I.7}
\|\mathcal{H}_\beta^\nu(x,y)\|_{E_\rho}
\lesssim 
\left\{
\begin{array}{ll}
\dfrac{y^{\nu+1/2}}{x^{\nu+3/2}},      & 0<y \leq x/2, \\
& \\
|x-y|^{-1},     & x/2<y \leq \min\{1,3x/2\},\\
& \\
\dfrac{x^{\nu+1/2}}{y^{\nu+3/2}},      & \min\{1,3x/2\}<y \leq 1,
\end{array}
\right.
\end{equation}
and
\begin{equation}\label{I.8}
      \|\partial_x\mathcal{H}_\beta^\nu(x,y)\|_{E_\rho}+  \|\partial_y\mathcal{H}_\beta^\nu(x,y)\|_{E_\rho}
      \lesssim \frac{1}{|x-y|^2}, \:\;x\neq y.
\end{equation}
By using \eqref{I.7} and 
\eqref{I.8} instead of the estimates in Lemmas \ref{Prop:2.1} and \ref{Prop:2.2}, we can proceed as in the proof of Theorem \ref{Th:1.1} and we can prove the results in Theorem \ref{Th:1.3}. We need to use the $L^p$ boundedness properties of Hardy type operators stated in \cite[Lemmas 1 and 2]{BHNV}.

\begin{Rem}
$L^p$ boundedness properties established in Theorem \ref{Th:1.3} can be seen as a pencil phenomenon like the one described by Mac\'ias, Segovia y Torrea \cite{MST}, see also \cite{NSpencil}.
\end{Rem}

\section{Proof of Theorem \ref{Th:1.2}}\label{S4}

\subsection{Boundedness of 
$\mathcal{V}_\rho(\{  P_t^\nu\}_{t>0})$  from
$H^1((0,1),\Delta_\nu)$ into $L^1((0,1),x^{2 \nu + 1}dx)$} \label{Sect4.1}
\quad \\

An atomic characterization of the space $H^1((0,1),\Delta_\nu)$ was established in \cite[Theorem A]{DPRS} when $\nu>-1/2$ and in \cite[Theorem 4.26]{BDL} when $\nu>-1$.\\

For every $j\in\NN_0:=\NN\cup\{0\}$, we consider the interval $I_j:=(1-2^{-j}, 1-2^{-j-1}].$ 
A complex valued function $a$ is said to be a $\Delta_\nu$-atom when it satisfies one of the following conditions:\\

\begin{itemize}
    \item[(a)]There exists an interval $I\subseteq (0,1)$ such that $\supp a\subseteq I$, $\|a\|_\infty\le m_\nu(I)^{-1}$, and 
    $$\int_Ia(x) x^{2\nu +1}dx=0.$$
    Here, as above, $m_\nu$ denotes the measure on $(0,\infty)$ having $x^{2\nu+1}$ as density function with respect to the Lebesgue measure in $(0,1).$\\
    
    \item[(b)] $a(x)=m_\nu(I_j)^{-1}\chi_{I_j}, $ for some $j\in\NN_0.$\\
    \end{itemize}
    
    In \cite[Theorem A]{DPRS} and \cite[Theorem 4.26]{BDL}  it was proved that $f\in L^1((0,1), x^{2\nu+1}dx)$ is in $H^1((0,1),\Delta_\nu)$ if, and only if, $$f=\sum_{i=0}^\infty \lambda_i a_i
    \quad \text{in} \quad  
    L^1((0,1), x^{2\nu+1}dx)$$
    where, for every $i\in\NN_0$, $a_i$ is a $\Delta_\nu$-atom and $\lambda_i>0,$ satisfying $\sum_{i=0}^\infty \lambda_i<\infty.$ Furthermore, if $f\in  H^1((0,1),\Delta_\nu)$, then
    $$
    \|f\|_{L^1((0,1), x^{2\nu+1}dx)}+\|P_*^\nu f\|_{L^1((0,1), x^{2\nu+1}dx)}
    \sim 
    \inf\sum_{i=0}^\infty \lambda_i,
    $$
    where the infimum is taken over all sequences $\{\lambda_i\}_{i=0}^\infty $ in $(0,\infty)$ such that $f=\sum_{i=0}^\infty\lambda_i a_i$ in $L^1((0,1),x^{2\nu+1}dx)$, where $a_i$ is a $\Delta_\nu$-atom for every $i\in\NN_0.$\\
    Since from Theorem \ref{Th:1.1} we know that $\mathcal{V}_\rho(\{P_t^\nu\}_{t>0})$ is bounded from $L^1((0,1),x^{2\nu+1}dx)$ into $L^{1,\infty}((0,1),x^{2\nu+1}dx)$, in order to see that $\mathcal{V}_\rho(\{P_t^\nu\}_{t>0})$ is bounded from $H^1((0,1),$ $\Delta_\nu)$ into $L^1((0,1),$ $x^{2\nu+1}dx)$ it is enough to show the following property.

\begin{Prop}\label{Prop:4.1}
Let $\rho>2$ and $\nu>-1$.
Then, for every $\Delta_\nu$-atom $a$, we have that
\begin{equation}\label{eq:4.0}
        \| \mathcal{V}_\rho(\{P_t^\nu\}_{t>0})(a)\|_{L^1((0,1),x^{2\nu+1}dx) }
        \lesssim 1,
\end{equation}
where the constant does not depend on $a.$
\end{Prop}

\begin{proof}
Along the proof, for $\alpha>0$ and ${\rm  I}$ an interval,  we shall use the notation $\alpha {\rm  I}$ to denote the interval with the same center as ${\rm I}$ and such that $|\alpha {\rm  I}|=\alpha|{\rm I}|$.

As in the proof of \cite[Lemma 4.33]{BDL}, we consider several cases.\\

\underline{\textit{Case $I)$}}: Suppose firstly that $a=m_\nu(I_j)^{-1}\chi_{I_j}$, for some $j\in\NN_0$.\\

Recall that the operator $\mathcal{V}_\rho(\{P_t^\nu\}_{t>0})$ is bounded from $L^2((0,1), x^{2\nu+1}dx)$ into itself (see Theorem \ref{Th:1.1}).\\

\underline{\textit{Case $I.1)$}}:  If $j=0,1,2$ we have that
    \begin{align}\label{eq:caseI1}
         \| \mathcal{V}_\rho(\{P_t^\nu\}_{t>0})(a)\|_{L^1((0,1),x^{2\nu+1}dx) }
         &\lesssim \| \mathcal{V}_\rho(\{P_t^\nu\}_{t>0})(a)\|_{L^2((0,1),x^{2\nu+1}dx) } \nonumber \\ 
         &\lesssim \|a\|_{L^2((0,1),x^{2\nu+1}dx) }
          \lesssim 
          m_\nu(I_j)^{-1/2}
          \lesssim 1.
    \end{align} \quad 

\underline{\textit{Case $I.2)$}}:    Assume now that $j\ge 3.$
Observe that $(0,1)\setminus (3I_j)=(0,1-3\cdot 2^{-j-1}].$ 
    By proceeding as in \eqref{eq:caseI1}, we get
  \begin{align}\label{eq:caseI2}
         \| \mathcal{V}_\rho(\{P_t^\nu\}_{t>0})(a)\|_{L^1(3I_j,x^{2\nu+1}dx) }
         &\le \| \mathcal{V}_\rho(\{P_t^\nu\}_{t>0})(a)\|_{L^2((0,1),x^{2\nu+1}dx) }m_\nu(3I_j)^{1/2} \nonumber \\
         &\lesssim
         \|a\|_{L^2((0,1),x^{2\nu+1}dx) }m_\nu(3I_j)^{1/2}\lesssim
         \frac{m_\nu(3I_j)^{1/2}}{m_\nu(I_j)^{1/2}}
         \lesssim 1.
    \end{align}
    On the other hand, observe that if $y\in I_j,$ then $y\sim 1$ and $1-y\sim 2^{-j},$ and if $x\in (0,1)\setminus (3I_j)$, then $|x-y|\sim 1-x$. According to \eqref{eq:2.1}, we obtain
    \begin{align}\label{eq:caseI2c}
        \int_{(0,1)\setminus (3I_j)} &\mathcal{V}_\rho(\{P_t^\nu\}_{t>0})(a)(x)x^{2\nu+1}dx \nonumber \\
        & \le \frac{1}{m_\nu(I_j)} \int_{(0,1)\setminus (3I_j)}\int_0^\infty \int_{I_j}|\partial_t P_t^\nu(x,y)|
        y^{2\nu+1}
        dy\;dt\; x^{2\nu+1}dx \nonumber \\
        & \lesssim
        \frac{1}{m_\nu(I_j)}\int_{(0,1)\setminus (3 I_j)}\int_{I_j}\int_0^\infty\int_0^\infty\frac{e^{-ct^2/u}}{u^{3/2}}W_u^\nu(x,y) du\; dt\; y^{2\nu+1}dy\; x^{2\nu+1}dx \nonumber \\
        &\lesssim
        \frac{1}{m_\nu(I_j)}\int_{(0,1)\setminus (3 I_j)}\int_{I_j}\int_0^\infty\frac{W_u^\nu(x,y)}{u} du\; y^{2\nu+1}dy \;x^{2\nu+1}dx \nonumber \\
        &\lesssim
        \frac{1}{m_\nu(I_j)}\int_{(0,1)\setminus (3 I_j)}\int_{I_j}\int_0^\infty\frac{(1+u)^{\nu+2}}{(u+xy)^{\nu+1/2}}\frac{(1-x)(1-y)
        }{u^{5/2}} \nonumber \\
        & \hspace{4cm} \times 
        \exp\Big(-c\frac{|x-y|^2}{u}-\lambda_{1,\nu}^2 u\Big)
        du\;  y^{2\nu+1}dy\; x^{2\nu+1}dx \nonumber \\
        &\lesssim
        \;2^{-j}\int_{(0,1)\setminus (3 I_j)}\int_0^\infty\frac{e^{-cu}(1-x)
        e^{-c|1-x|^2/u}}{u^{5/2}(u+x)^{\nu+1/2}} du\; x^{2\nu+1}dx.
    \end{align}
    Now we distinguish two situations. If $\nu\ge -1/2$, we have that
    \begin{align}\label{eq:caseI2c1}
        \int_{(0,1)\setminus (3 I_j)}&\int_0^\infty\frac{e^{-cu}(1-x)
        e^{-c|1-x|^2/u}}{u^{5/2}(u+x)^{\nu+1/2}} du\; x^{2\nu+1}dx \nonumber \\
        &\lesssim
        \int_{(0,1)\setminus (3 I_j)}x^{\nu+1/2}(1-x)\int_0^\infty\frac{e^{-c|1-x|^2/u}}{u^{5/2}} du\;dx \nonumber \\
        &\lesssim
        \int_{(0,1)\setminus (3 I_j)}\frac{x^{\nu+1/2}}{(1-x)^2}dx
        \lesssim
        \int_0^{1-3\cdot 2^{-j-1}}\frac{dx}{(1-x)^2}
        \lesssim
        2^j.
    \end{align}
    Next, suppose that $-1<\nu<-1/2.$ Since $$(u+x)^{-\nu-1/2}
    \lesssim u^{-\nu-1/2}+x^{-\nu-1/2},
    \quad u,x\in (0,\infty),$$
    we get 
    \begin{align}\label{eq:caseI2c2}
         \int_{(0,1)\setminus (3 I_j)}&\int_0^\infty\frac{e^{-cu}(1-x)e^{-c|1-x|^2/u}}{u^{5/2}(u+x)^{\nu+1/2}} du\; x^{2\nu+1}dx \nonumber \\
        &\lesssim
        \int_{(0,1)\setminus (3 I_j)}\int_0^\infty\frac{e^{-cu}(1-x)(u^{-\nu-1/2}+x^{-\nu-1/2})e^{-c|1-x|^2/u}}{u^{5/2}} du\;x^{2\nu+1}dx \nonumber \\
        &\lesssim
        \int_0^{1-3\cdot 2^{-j-1}}\frac{x^{\nu+1/2}}{(1-x)^2}dx+ \int_0^{1-3\cdot 2^{-j-1}}\frac{x^{2\nu+1}}{(1-x)^2}dx \nonumber \\
        &\lesssim
        \int_0^{1/2}x^{2\nu+1} dx+\int_{1/2}^{1-3\cdot 2^{-j-1}}\frac{dx}{(1-x)^2}
        \lesssim 2^j.
    \end{align}
    Combining \eqref{eq:caseI2c}, \eqref{eq:caseI2c1} and \eqref{eq:caseI2c2}, we have that
    $$
    \int_{(0,1)\setminus (3 I_j)}\mathcal{V}_\rho(\{P_t^\nu\}_{t>0})(a)(x) x^{2\nu+1}dx \lesssim 1,
    $$
    and together with \eqref{eq:caseI2} we conclude
    $$
    \| \mathcal{V}_\rho(\{P_t^\nu\}_{t>0})(a)\|_{L^1((0,1), x^{2\nu+1}dx)} \leq C,
    $$
    where $C>0$ does not depend on $a$.\\

\underline{\textit{Case $II)$}}:
 
Suppose that $a$ is a $\Delta_\nu$-atom such that $\supp a\subseteq I$, where $I$ is an interval contained in $I_j^*=\frac{{21}}{20}I_j\cap(0,1)$, for some $j\in\NN_0$,
$\|a\|_{L^\infty((0,1), x^{2\nu+1}dx)}\le m_\nu(I)^{-1}$ and
$$\int_I a(x)x^{2\nu+1}dx=0.$$
    We denote by $x_I$ the center of $I$ and by $r_I$ the radius of $I$.\\
    
\underline{\textit{Case $II.1)$}}: Assume firstly that $j\ge 1$.\\

By using once again that $\mathcal{V}_\rho(\{P_t^\nu\}_{t>0})$ is bounded from $L^2((0,1), x^{2\nu+1}dx)$ into itself, we obtain
\begin{align}\label{eq:caseII1_2I}
\| \mathcal{V}_\rho(\{P_t^\nu\}_{t>0})(a)\|_{L^1(2I, x^{2\nu+1}dx)}
&\le   \| \mathcal{V}_\rho(\{P_t^\nu\}_{t>0})(a)\|_{L^2((0,1), x^{2\nu+1}dx)}m_{\nu}(2I)^{1/2} \nonumber \\
&\lesssim \|a\|_{L^2((0,1), x^{2\nu+1}dx)}m_{\nu}(2I)^{1/2}
\lesssim \frac{m_{\nu}(2I)^{1/2}}{m_{\nu}(I)^{1/2}}\lesssim 1.
\end{align}
Moreover, since
$$\int_I a(x)x^{2\nu+1}dx=0,$$ 
we can write
\begin{align*}
&\| \mathcal{V}_\rho(\{P_t^\nu\}_{t>0})(a)\|_{L^1(3I_j\setminus (2I), x^{2\nu+1}dx)}  \\
& \quad \le 
\int_{3I_j\setminus (2I)}\int_I\int_0^\infty |\partial_t(P_t^\nu(x,y)-P_t^\nu(x,x_I))|\;dt\;|a(y)|y^{2\nu+1}dy\; x^{2\nu+1}dx\\
&\quad \lesssim
\int_{3I_j\setminus (2I)}\int_I\int_0^\infty\int_0^\infty |\partial_t(te^{-t^2/(4u)})||W_u^\nu(x,y)-W_u^\nu(x,x_I)|\frac{du}{u^{3/2}} dt \\
& \hspace{4cm} \times |a(y)|y^{2\nu+1}dy\; x^{2\nu+1}dx\\
&\quad \lesssim
\int_{3I_j\setminus (2I)}\int_I\int_0^\infty \frac{|W_u^\nu(x,y)-W_u^\nu(x,x_I)|}{u^{3/2}} \int_0^\infty e^{-c t^2/u}
dt\;{du}\;|a(y)|y^{2\nu+1}dy\; x^{2\nu+1}dx\\
&\quad \lesssim
\int_{3I_j\setminus (2I)}\int_I\int_0^\infty \frac{|W_u^\nu(x,y)-W_u^\nu(x,x_I)|}{u} {du}\;|a(y)|y^{2\nu+1}dy\; x^{2\nu+1}dx.
\end{align*}
        Observe that if $x\in 3 I_j$ and $z\in I$, then $x\sim z\sim 1$, while if $x\in (0,1)\setminus (2I)$ and $z\in I$, then $|x-z|\sim |x-x_I|$.   Therefore, according to \eqref{eq:1.2}, we get that
\begin{align*}
|W_u^\nu(x,y)-W_u^\nu(x,x_I)|&\le 
\Big| \int_{x_I}^y\partial_zW_u^\nu(x,z)dz\Big|
\le
\Big| \int_{x_I}^y \frac{e^{-c (x-z)^2/u}}{(xz)^{\nu+1/2}u}dz\Big|\\
& \lesssim
|x_I-y|\frac{
e^{-c(x-x_I)^2/u}
}{u}, \:\:\:x\in (3I_j)\setminus (2I) \text{ and } y\in I. 
\end{align*}
It follows that
\begin{align}\label{eq:caseII1_3Ij-2I}
&\| \mathcal{V}_\rho(\{P_t^\nu\}_{t>0})(a)\|_{L^1(3I_j\setminus (2I), x^{2\nu+1}dx)} \nonumber \\
& \qquad \lesssim
\int_{3I_j\setminus (2I)}\int_I\int_0^\infty |x_I-y|\frac{e^{-c (x-x_I)^2/u}}{u^2} {du}\;|a(y)|y^{2\nu+1}dy\; x^{2\nu+1}dx \nonumber \\
&\qquad 
\lesssim
\int_{3I_j\setminus (2I)}\int_I |x_I-y|\frac{|a(y)|}{(x-x_I)^2}y^{2\nu+1}dy\; x^{2\nu+1}dx \nonumber \\
&\qquad
\lesssim
\int_{3I_j\setminus (2I)} \frac{r_I}{(x-x_I)^2}\; dx \lesssim 
 r_I\int_{(0,1)\setminus (2I)}\frac{dx}{(x-x_I)^2}
\lesssim 1.
\end{align}
In a similar way, using that $|x-x_I|\sim |1-x|$ for $x\in (0,1)\setminus (3I_j)$, we get 
\begin{align}\label{eq:caseII1_3Ij}
& \| \mathcal{V}_\rho(\{P_t^\nu\}_{t>0})(a)\|_{L^1((0,1)\setminus (3I_j), x^{2\nu+1}dx)} \nonumber \\
&\qquad \lesssim \int_{(0,1)\setminus (3I_j)}\frac{r_I}{|x-x_I|^2}\;  x^{\nu+1/2}dx 
\lesssim
\int_0^{1-3\cdot 2^{-j-1}}\frac{2^{-j}}{|1-x|^2}x^{\nu+1/2}dx \nonumber \\
& \qquad \lesssim
2^{-j}\Big(\int_0^{1/4}x^{\nu+1/2}dx +\int_{1/4}^{1-3\cdot 2^{-j-1}}\frac{1}{|1-x|^2}dx\Big)
\lesssim 1.
\end{align}
Putting together 
\eqref{eq:caseII1_2I},
\eqref{eq:caseII1_3Ij-2I} and
\eqref{eq:caseII1_3Ij}, we deduce
             $$
    \| \mathcal{V}_\rho(\{P_t^\nu\}_{t>0})(a)\|_{L^1((0,1), x^{2\nu+1}dx)}\le C,
    $$
    where $C>0$ does not depend on the atom $a$.\\
 
\underline{\textit{Case $II.2)$}}: Consider now the situation $j=0$. For $x\in (0,1)$, define the operators 
        $$
    \mathcal{V}_\rho^{(1)}(\{P_t^\nu\}_{t>0})(f)(x)
:=
\sup_{\substack{1\le t_1<\dots<t_k\\k\in\NN}}
\Big(
\sum_{j=1}^{k-1}
\Big| 
P^\nu_{t_j}(f)(x)
-
P^\nu_{t_{j+1}}(f)(x)
\Big|^\rho
\Big)^{1/\rho}, 
    $$
  and
    $$
    \mathcal{V}_\rho^{(2)}(\{P_t^\nu\}_{t>0})(f)(x)
:=\sup_
{\substack{0< t_1<\dots<t_k\le 1\\k\in\NN}}
\Big(
\sum_{j=1}^{k-1}
\Big| 
P^\nu_{t_j}(f)(x)
-
P^\nu_{t_{j+1}}(f)(x)
\Big|^\rho
\Big)^{1/\rho},  
    $$
and observe that
\begin{equation}\label{eq:decompVp}
  \mathcal{V}_\rho(\{P_t^\nu\}_{t>0})(f)\le \mathcal{V}_\rho^{(1)}(\{P_t^\nu\}_{t>0})(f)+ \mathcal{V}_\rho^{(2)}(\{P_t^\nu\}_{t>0})(f).
\end{equation}

\quad

We start the analysis with the operator $\mathcal{V}_\rho^{(1)}$. For $x \in (0,1)$, we have that
\begin{align}\label{eq:Vpalpha}
& \mathcal{V}_\rho^{(1)}(\{P_t^\nu\}_{t>0})(a)(x) \nonumber \\
&\qquad 
\le 
\int_1^\infty\int_0^1|\partial_tP_t^\nu(x,y)||a(y)|y^{2\nu+1}dydt \nonumber \\
&\qquad \lesssim 
\int_0^1 \int_1^\infty\int_0^\infty\frac{|\partial_t(t
e^{-t^2/(4u)}
)|}{u^{3/2}}W_u^\nu(x,y)\;dudt\;|a(y)|y^{2\nu+1}dy \nonumber \\
& \qquad \lesssim 
\int_0^1 \int_1^\infty\int_0^\infty\frac{e^{-c t^2/u}}{u^{3/2}}W_u^\nu(x,y)\;dudt\;|a(y)|y^{2\nu+1}dy \nonumber \\
&\qquad \lesssim  
\int_0^1 \int_1^\infty\frac{1}{t^{3+\alpha}}\int_0^\infty u^{\alpha/2}W_u^\nu(x,y)\;dudt\;|a(y)|y^{2\nu+1}dy \nonumber \\
& \qquad \lesssim 
\int_0^1|a(y)|y^{2\nu+1}\int_0^\infty \frac{(1+u)^{\nu+2}}{(u+xy)^{\nu+1/2}}
\exp\Big(-c\frac{(x-y)^2}{u}-\lambda_{1,\nu}^2u\Big)
u^{\alpha/2-1/2} \, du \, dy.
\end{align}
for any $\alpha\ge 0.$\\

If $\nu\ge -1/2$, by taking $\alpha=2\nu+1$
in \eqref{eq:Vpalpha}, we obtain
\begin{align*}
      \mathcal{V}_\rho^{(1)}(\{P_t^\nu\}_{t>0})(a)(x)&\lesssim \int_0^1|a(y)|y^{2\nu+1}dy\int_0^\infty \frac{e^{-c\lambda_{1,\nu}^2u}}{\sqrt{u}}du \\
      &\lesssim 
      \|a\|_{L^\infty((0,1), x^{2\nu+1}dx)}m_\nu(I_0)
      \lesssim 1,\:\:\;x\in(0,1).
\end{align*}

However, if $-1<\nu<-1/2$, we choose $\alpha=0$ in \eqref{eq:Vpalpha} and write
\begin{align*}
& \mathcal{V}_\rho^{(1)}(\{P_t^\nu\}_{t>0})(a)(x)\\
& \qquad \lesssim 
\int_0^1|a(y)|y^{2\nu+1}\int_0^\infty (1+u)^{\nu+2}(u^{-\nu-1/2}+(xy)^{-\nu-1/2})\frac{e^{-c\lambda_{1,\nu}^2u}}{\sqrt{u}}du dy\\
&\qquad \lesssim 
\int_0^1|a(y)|y^{2\nu+1}dy
\lesssim 1,\:\:\;x\in(0,1).
\end{align*}
Hence, we conclude  
\begin{equation}\label{eq:Vp1}
\|\mathcal{V}_\rho^{(1)}(\{P_t^\nu\}_{t>0})(a)\|_{L^1((0,1),x^{2\nu+1}dx)}
\lesssim 1.
\end{equation}

\quad 

Next we study the operator $\mathcal{V}_\rho^{(2)}$. 
Let $I_0^{**}=\left(\frac{21}{20}\right)^2 I_0\cap (0,1)$.
If $x\in (0,1)\setminus I_0^{**}$ and $y\in I_0^*$, then $|x-y|\sim x\sim 1$. By considering as above
the situations of $\nu\ge -1/2$ and $-1<\nu<-1/2$ separately, we obtain that
\begin{align}\label{eq:Vp2c}
&\|\mathcal{V}_\rho^{(2)}(\{P_t^\nu\}_{t>0})(a)\|_{L^1((0,1)\setminus I_0^{**},x^{2\nu+1}dx)} \nonumber \\
&\qquad \lesssim 
\int_{(0,1)\setminus I_0^{**}}\int_I |a(y)|y^{2\nu+1}\int_0^1\int_0^\infty \frac{|\partial_t(te^{-t^2/(4u)})|}{u^{3/2}}W_u^\nu(x,y)\;dudt\;dy\;x^{2\nu+1}dx \nonumber \\
&\qquad \lesssim 
\int_{(0,1)\setminus I_0^{**}}\int_I |a(y)|y^{2\nu+1}\int_0^1\int_0^\infty\frac{e^{-c t^2/u}}{u^{3/2}}\frac{(1+u)^{\nu+2}}{(u+xy)^{\nu+1/2}} \\
& \hspace{7cm} \times 
\frac{
e^{-c(x-y)^2/u-cu}
}{\sqrt{u}}du\;dt\;dy\;x^{2\nu+1}dx \nonumber \\
&\qquad \lesssim 
\int_{(0,1)\setminus I_0^{**}}\int_I |a(y)|y^{2\nu+1}\int_0^1\int_0^\infty\frac{(1+u)^{\nu+2}}{u^2}\frac{e^{-c/u}e^{-cu}}{(u+y)^{\nu+1/2}}du\;dt\;dy\;x^{2\nu+1}dx \nonumber \\
&\qquad \lesssim 
\int_I |a(y)|y^{2\nu+1}dy
\lesssim 1.
\end{align}

\quad 

It remains to verify the uniform estimate
\begin{equation}\label{eq:Vp2}
\|\mathcal{V}_\rho^{(2)}(\{P_t^\nu\}_{t>0})(a)\|_{L^1(I_0^{**},x^{2\nu+1}dx)}
\lesssim 1.
\end{equation}
Its proof is more subtle and it is convenient to introduce the following auxiliary operator 
    $$
    \mathfrak{W}_t^\nu(f)(x)
    :=\int_0^\infty \mathfrak{W}_t^\nu(x,y)f(y)y^{2\nu+1}dy, \quad x,t>0,
    $$
where
$$
\mathfrak{W}_t^\nu(x,y)
:=\frac{(xy)^{-\nu}}{2t}I_\nu
\Big( \frac{xy}{2t}\Big)
e^{-(x^2+y^2)/4t}
$$

and
$I_\nu$ denotes the modified Bessel function of the first kind and order $\nu$.
$\{ \mathfrak{W}_t^\nu \}_{t>0}$ is a symmetric diffusion semigroup associated with
$\Delta_\nu$ on $((0,\infty),$ $x^{2\nu+1}dx). $
As in \cite{DPRS}, we consider a function $\rho\in C^\infty(0,\infty)$ such that $\rho(x)=1$ for $x\in I_0^{**}$ and $\rho(x)=0$ for $x\not\in I_0^{***}$, {being $I_0^{***}=\left(\frac{21}{20}\right)^3I_0\cap (0,1)$}. By using \cite[Equation (4.6)]{DPRS}, we can write
\begin{align*}
{W}_t^\nu(x,y)-\mathfrak{W}_t^\nu(x,y)&=\rho(x){W}_t^\nu(x,y)-\rho(y)\mathfrak{W}_t^\nu(x,y)\\
&=\int_0^t\int_0^\infty \mathfrak{W}_{t-s}^\nu(x,z)\rho''(z){W}_s^\nu(z,y)z^{2\nu+1}dzds\\
&\qquad 
+2\int_0^t\int_0^\infty \partial_z\mathfrak{W}_{t-s}^\nu(x,z)\rho'(z){W}_s^\nu(z,y)z^{2\nu+1}dzds\\
&\qquad 
+\int_0^t\int_0^\infty \mathfrak{W}_{t-s}^\nu(x,z)\rho'(z)\frac{2\nu+1}{z}{W}_s^\nu(z,y)z^{2\nu+1}dzds\\
&= : 
\sum_{i=1}^3 R_t^i(x,y), \quad x,y\in I_0^{**}.
\end{align*}
By proceeding as in the proof of \cite[Lemma 4.6]{DPRS} we can show that
$$
|R_t^1(x,y)|
\lesssim 
\int_0^t \frac{e^{-c/(t-s)}}{(t-s)^{\nu+1}}\frac{e^{-c/s}}{s^{\nu+1}}ds, \quad x,y\in I_0^{**} \text{ and } t\in (0,1).
$$
Since $\nu>-1$, it follows that
\begin{equation*}
    |R_t^1(x,y)|
    \lesssim 
    t, \quad x,y\in I_0^{**} \text{ and } t\in (0,1).
\end{equation*}
In a similar way we can see that, for $i=2,3$,
$$
|R_t^i(x,y)|
\lesssim 
t, \quad x,y\in I_0^{**} \text{ and } t\in (0,1).
$$
Hence, 
\begin{equation}\label{eq:4.1}
|{W}_t^\nu(x,y)-\mathfrak{W}_t^\nu(x,y)|
\lesssim 
t, \quad x,y\in I_0^{**} \text{ and } t\in (0,1).
\end{equation}
We define, for every $x,t>0$,
$$
\mathfrak{P}_t^\nu(f)(x)
:=\frac{t}{2\sqrt{\pi}}\int_0^\infty\frac{e^{-t^2/(4u)}}{u^{3/2}} \mathfrak{W}_t^\nu f(x)du=\int_0^1 \mathfrak{P}_t^\nu(x,y)f(y)y^{2\nu+1}dy, 
    $$
where
$$
\mathfrak{P}_t^\nu(x,y)
:=\frac{t}{2\sqrt{\pi}}\int_0^\infty\frac{e^{-t^2/(4u)}}{u^{3/2}} \mathfrak{W}_t^\nu(x,y)du.
$$
For $x \in (0,1)$, we have that 
\begin{equation}\label{eq:VrPP}
|\mathcal{V}_\rho^{(2)}(\{P_t^\nu\}_{t>0})(a)(x)|
\le 
\mathcal{V}_\rho^{(2)}(\{P_t^\nu-\mathfrak{P}_t^\nu\}_{t>0})(a)(x)
+
\mathcal{V}_\rho^{(2)}(\{\mathfrak{P}_t^\nu\}_{t>0})(a)(x).
\end{equation}
Furthermore,
\begin{align*}
& \mathcal{V}_\rho^{(2)}(\{P_t^\nu-\mathfrak{P}_t^\nu\}_{t>0})(a)(x)\\
& \qquad \leq
\int_I |a(y)|\int_0^1\int_0^\infty \frac{|\partial_t(te^{-t^2/(4u)})|}{u^{3/2}}|W_u^\nu(x,y)-\mathfrak{W}_u^\nu(x,y)|du\;dt\;y^{2\nu+1}dy\\
& \qquad \lesssim
\int_I |a(y)|\int_0^1\int_1^\infty \frac{|W_u^\nu(x,y)|}{u^{3/2}}du\;dt\;y^{2\nu+1}dy\\
& \qquad \qquad + 
\int_I |a(y)|\int_0^1\int_1^\infty \frac{|\mathfrak{W}_u^\nu(x,y)|}{u^{3/2}}du\;dt\;y^{2\nu+1}dy\\
& \qquad \qquad+
\int_I |a(y)|\int_0^1\int_0^1 \frac{|W_u^\nu(x,y)-\mathfrak{W}_u^\nu(x,y)|}{u^{3/2}}du\;dt\;y^{2\nu+1}dy \\
& \qquad =:
I_1(x)+I_2(x)+I_3(x),\quad x\in (0,1).
\end{align*}
By analysing separately the cases $\nu\ge -1/2$ and $-1<\nu<-1/2$, we obtain
\begin{align*}
I_1(x)
& \lesssim 
\int_I |a(y)|y^{2\nu+1}\int_0^1\int_1^\infty\frac{e^{-c t^2/u}}{u^{3/2}}\frac{(1+u)^{\nu+2}}{(u+xy)^{\nu+1/2}}\frac{e^{-c(x-y)^2/u-cu}}{\sqrt{u}}du\;dt\;dy\;\\
& \lesssim
\int_I |a(y)|y^{2\nu+1}\int_0^1\int_1^\infty e^{-cu}\;du\;dt\;dy
\lesssim 1, \quad x\in I_0^{**}.
\end{align*}
Also, \eqref{eq:4.1} leads to
\begin{align*}
       I_3(x)&
       \lesssim
       \int_I |a(y)|y^{2\nu+1}\int_0^1\int_0^1\frac{1}{u^{1/2}}du\;dt\;dy
       \lesssim 1, \quad x\in I_0^{**}.
\end{align*}
On the other hand, according \cite[Lemma 4.5]{DPRS}, we have that
 \begin{equation*}
0\le \mathfrak{W}_t^\nu(x,y)
     \lesssim \frac{1}{m_\nu(B(x,\sqrt{t}))}
     e^{-c (x-y)^2/t},\qquad x,y\in (0,\infty) \text{ and } t>0,
 \end{equation*}
 where 
 \begin{equation}\label{eq:4.3}
     m_\nu(B(x,\sqrt{t}))\sim\left\{\begin{array}{c}
     x^{2\nu+1}\sqrt{t}, \quad \sqrt{t}\le x\\
      (\sqrt{t})^{2\nu+2}, \quad \sqrt{t}> x\\
     \end{array}\right.
 \end{equation}
 see \cite[Equation (4.7)]{DPRS}.
Then, 
\begin{align*}
I_2(x)
&\lesssim
\int_I |a(y)|y^{2\nu+1}\int_0^1\int_1^\infty\frac{1}{u^{3/2}m_\nu(B(x,\sqrt{u})}du\;dt\;dy\\
&\lesssim
\int_I |a(y)|y^{2\nu+1}\int_0^1\int_1^\infty\frac{1}{u^{5/2+\nu}}du\;dt\;dy
\lesssim 1, \quad x\in I_0^{**}.
\end{align*}
We have obtained that
\begin{equation}\label{eq:VpP-P}
\mathcal{V}_\rho^{(2)}(\{P_t^\nu-\mathfrak{P}_t^\nu\}_{t>0})(a)(x) \lesssim 1, \quad x\in I_0^{**}. \end{equation}

\quad 

On the other hand, it is clear that
\begin{equation}\label{eq:VpPP1}
\mathcal{V}_\rho^{(2)}(\{\mathfrak{P}_t^\nu\}_{t>0})(a)(x)\le \mathcal{V}_\rho(\{\mathfrak{P}_t^\nu\}_{t>0})(a)(x),\quad x\in (0,\infty).
\end{equation}

\quad 

Then, if $\nu>-1/2$, according to \cite[Theorem 1.3]{WYZ}, we get
\begin{equation}\label{eq:VpPP2a}
\| \mathcal{V}_\rho(\{\mathfrak{P}_t^\nu\}_{t>0})(a)\|_{L^1((0,\infty), x^{2\nu+1}dx)}
\lesssim 1.
\end{equation}

\quad

Suppose now that $-1<\nu\le -1/2.$ We are going to see that 
$$
|\partial_t\partial_x \mathfrak{P}_t^\nu(x,y)|
\lesssim \frac{1}{m_\nu(B(x,|x-y|))(t+|x-y|)^2}, \quad x,y\in (0,1).
$$
Indeed, since
$$
\partial_t\partial_x\mathfrak{P}_t^\nu(x,y)=\frac{1}{2\sqrt{\pi}}\int_0^\infty\frac{\partial_t(te^{-t^2/(4u)})}{u^{3/2}} \partial_x \mathfrak{W}_t^\nu(x,y)du,\quad x,y,t\in (0,\infty),
$$
according to \cite[Lemma 4.5 (b)]{DPRS}  and \eqref{eq:4.3}, we get
\begin{align*}
|\partial_t\partial_x\mathfrak{P}_t^\nu(x,y)|
&\lesssim
\int_0^\infty \frac{
e^{-c(t^2+(x-y)^2)/u}
}{u^{2}m_\nu(B(x,\sqrt{u}))}du\\
&\lesssim 
\int_0^{x^2} \frac{e^{-c(t^2+(x-y)^2)/u}}{u^{5/2}x^{2\nu+1}}du+ \int_{x^2}^\infty \frac{e^{-c(t^2+(x-y)^2)/u}}{u^{\nu+3}}du, \\
&\lesssim
\frac{1}{(t^2+|x-y|^2)^{3/2}x^{2\nu+1}}+  \frac{\chi_{\{|x-y|\le  x\}}}{(t^2+|x-y|^2)x^{2\nu+2}}
+\frac{\chi_{\{x>|x-y|\} }}{(t^2+|x-y|^2)^{\nu+2}}\\
& \lesssim
\frac{\chi_{\{|x-y|\le  x\}}}{(t+|x-y|)^2|x-y|x^{2\nu+1}}
+\frac{\chi_{\{x>|x-y| \}}}{(t+|x-y|)^2|x-y|^{2\nu+2}}\\
&\lesssim
\frac{1}{m_\nu(B(x,|x-y|))(t+|x-y|)^2} \quad x,y\in (0,1), \:\; t>0.
\end{align*}
Now, by proceeding as in \cite[page 859 and 860]{WYZ} we can deduce that 
\begin{equation}\label{eq:VpPP2}
\|\mathcal{V}_\rho(\{\mathfrak{P}_t^\nu\}_{t>0})(a)\|_{L^1(I_0^{**},x^{2\nu+1}dx)} 
\lesssim 1.
\end{equation}

\quad 

Putting together the estimates
\eqref{eq:VrPP},
\eqref{eq:VpP-P},
\eqref{eq:VpPP1},
\eqref{eq:VpPP2a}
and
\eqref{eq:VpPP2}
we establish \eqref{eq:Vp2}.\\

Therefore, 
\eqref{eq:decompVp},
\eqref{eq:Vp1},
\eqref{eq:Vp2c}
and
\eqref{eq:Vp2}
yield to
$$
\|\mathcal{V}_\rho(\{{P}_t^\nu\}_{t>0})(a)\|_{L^1((0,1),x^{2\nu+1}dx)}\le C,
$$
where $C>0$ does not depend on the atom $a.$\\

\underline{\textit{Case $III)$}}: Finally assume that there is no $j\in\NN_0$ such that $\supp a\subseteq I_j^*$.\\

We choose an interval $I\subseteq (0,1)$ such that $\supp a\subseteq I$, $\|a\|_\infty\le m_\nu(I)^{-1}$ and 
    $$\int_I a(x)x^{2\nu+1}dx=0.$$
We are going to proceed as in \cite[Case 3 in the proof of Theorem A]{DPRS}. Take $N\in\NN_0$, $M\in\NN\cup\{\infty\}$ such that $\displaystyle I\subseteq\bigcup_{j=N}^MI_j^*$ and satisfying the following property: if $\tilde{N}\in\NN_0$ and $\tilde{M}\in\NN\bigcup\{ \infty\}$ and $\displaystyle I\subseteq\bigcup_{j=\tilde{N}}^{\tilde{M}}I_j^*$, then $\tilde{N}\le N$, $M\le \tilde{M}.$\\

Observe that $m_\nu (I)\sim 2^{-N}$ and $m_\nu(I_j^*)\sim  {2^{-j}}$, $j\in\NN_0$. We write     

    $$
    a=\sum_{j=N}^M 2^{N-j}b_j,
    $$
for 
\begin{equation*}
b_j
:=2^{j-N}a\chi_{I_j^*}, 
\quad j=N,\dots,M.    
\end{equation*}
We have that $\supp b_j\subseteq I_j^*$ and $$\|b_j\|_{L^\infty((0,1), x^{2\nu+1}dx)}
\lesssim m_\nu(I_j^*)^{-1}, \quad j=N,\dots,M,$$
where $C$ does not depend on $a$, that is, on $N$ and $M.$ By proceeding as in \cite[Remark 5.9]{DPRS} and using the properties shown in I) and II), we obtain 
\begin{equation*}
\|\mathcal{V}_\rho(\{{P}_t^\nu\}_{t>0})(a)\|_{L^1((0,1),x^{2\nu+1}dx)}
\lesssim
\sum_{j=N}^M {2^{N-j}}
\lesssim 1. \qedhere
\end{equation*}
\end{proof}


\subsection{Equivalence of the quantities
\eqref{eq:Th1.2a}
and
\eqref{eq:Th1.2b}}
\quad \\
It suffices to observe that, for every $t>0$,
\begin{equation}\label{eq:4.10}
    P_t^\nu(f)\le \mathcal{V}_\rho(\{{P}_t^\nu\}_{t>0})(f)+|P_1^\nu(f)|, \quad  f\in L^1((0,1),x^{2\nu+1}dx).
    \end{equation}
    Indeed, \eqref{eq:4.10} implies that
    \begin{equation*}
    P_*^\nu(f)\le \mathcal{V}_\rho(\{{P}_t^\nu\}_{t>0})(f)+|P_1^\nu(f)|,\quad  f\in L^1((0,1),x^{2\nu+1}dx),
    \end{equation*}
    and since $P_1^\nu$ is bounded from  $L^1((0,1),x^{2\nu+1}dx)$ into itself,  for $f\in L^1((0,1),x^{2\nu+1}dx)$ it follows that
\begin{equation*}
\|P_*^\nu(f)\|_{L^1((0,1),x^{2\nu+1}dx)}
\lesssim
\|\mathcal{V}_\rho(\{{P}_t^\nu\}_{t>0})(f)\|_{L^1((0,1),x^{2\nu+1}dx)}
+\|f\|_{L^1((0,1),x^{2\nu+1}dx)}. 
\end{equation*}

\section{Proof of Theorem \ref{Th:1.4}}
\label{Sect5}

The Hardy space $H^1((0,1),S_\nu)$ was characterised by using atoms in \cite[Theorem B]{DPRS} (see also \cite[Theorem 4.17]{BDL}). We consider as in \cite{BDL,DPRS} the intervals $I_j$, $j\in\mathbb{Z}\setminus\{0\}$, given by
$$
I_j:=
\left\{\begin{array}{cc}    (1-2^{-j},1-2^{-j-1}], & j\ge 1, \\
  (2^{j-1},2^j]   & j\le -1. 
\end{array}\right.
$$
A complex function $a$ defined on $(0,1)$ is said to be a $S_\nu$-atom when one of the following conditions is satisfied: \\
\begin{itemize}
    \item[(a)] There exists an interval $I\subseteq (0,1)$ for which $\supp a\subseteq I,$  $\|a\|_{L^\infty((0,1),dx)}
    \le  m(I)^{-1}$ and 
    $$\int_0^1 a(x)dx=0.$$
    
    \quad 
    
    \item[(b)] $a=m(I_j)^{-1} \chi_{I_j}$, for some $j\in\mathbb{Z}\setminus\{0\},$ \\
\end{itemize}
where $m$ denotes the Lebesgue measure on $(0,\infty).$\\

In \cite[Theorem B]{DPRS} (see also \cite[Theorem 4.17]{BDL}) it was proven that $f\in L^1((0,1),dx)$ is in $ H^1((0,1), S_\nu)$ if, and only if, 
$$f=\sum_{i=0}^\infty \lambda_i a_i
\quad \text{in} \quad  
L^1((0,1),dx),$$
where $\lambda_i>0$, $i\in\NN_0$, being $\sum_{i=0}^\infty\lambda_i<\infty,$ and $a_i$ is a $S_\nu$-atom, $i\in\NN_0.$ Furthermore, for every $f\in H^1((0,1), S_\nu)$,
$$
\|f\|_{L^1((0,1),dx)} + \|\mathcal{P}_*^\nu f\|_{L^1((0,1),dx)} \sim \inf \sum_{i=0}^\infty\lambda_i,
$$
where the infimum is taken over all sequences $\{\lambda_i\}_{i=0}^\infty$ in $(0,\infty)$ such that 
$\sum_{i=0}^\infty\lambda_i<\infty$ and $f=\sum_{i=0}^\infty \lambda_i a_i$  in $L^1((0,1),dx)$, where, for every $i\in\NN_0$, $a_i$ is a $S_\nu$-atom.\\

Since the variation operator $\mathcal{V}_\rho(\{\mathcal{P}_t^\nu\}_{t>0})$ is bounded from $L^1((0,1),dx)$ into $L^{1,\infty}((0,1),$ $dx)$, in order to see that $\mathcal{V}_\rho(\{\mathcal{P}_t^\nu\}_{t>0})$ is bounded from  $L^{1}((0,1), dx)$ into $H^1((0,1),S_\nu)$ it is sufficient to show the following uniform estimate.

\begin{Prop}
Let $\rho>2$ and {$\nu>-1/2$}.
Then, for every $S_\nu$-atom $a$, we have that
\begin{equation*}
\|\mathcal{V}_\rho(\{\mathcal{P}_t^\nu\}_{t>0})(a)\|_{L^1((0,1),dx)}
\lesssim 1,
\end{equation*}
where the constant does not depend on $a.$
\end{Prop}

\begin{proof}
We shall proceed as in the proof of Proposition \ref{Prop:4.1}, by using \cite[Lemma 4.19]{BDL}. 
Cases I) and II) when $j\ge 1$  can be developed analogously to the corresponding ones in the previous Section \ref{Sect4.1}.\\

We now suppose that $a$ is a $S_\nu$-atom such that $\supp a\subseteq I,$ $\|a\|_\infty\le m(I)^{-1}$ and $\int_0^1a(x)dx=0$, where $I$ is an interval contained in $I_j^*$, for some $j\le -1$.\\

Since $\mathcal{V}_\rho(\{\mathcal{P}_t^\nu\}_{t>0})$ is bounded from $L^2((0,1),dx)$ into itself, we get 
\begin{equation}\label{eq:Prop5.1_2I}
\|\mathcal{V}_\rho(\{\mathcal{P}_t^\nu\}_{t>0})(a)\|_{L^1(2I,dx)}\lesssim 1.
\end{equation}
Moreover, by using that $\displaystyle\int_0^\infty a(y)dy=0$, we also have that
\begin{align*}
&\|\mathcal{V}_\rho(\{\mathcal{P}_t^\nu\}_{t>0}(a)\|_{L^1((3I_j)\setminus(2I),dx)} \\
& \qquad \lesssim 
    \int_{(3I_j)\setminus(2I)}\int_I\int_0^\infty \frac{|\mathcal{W}_u^\nu(x,y)-\mathcal{W}_u^\nu(x,x_I)|}{u}du|a(y)|dydx.
\end{align*}

If $z\in 2I$, then $z\sim x_I$ and if $x\in (0,1)\setminus (2I)$ and $z\in I$, then $|x-z|\sim|x-x_I|.$
Moreover, since  
\begin{equation*}
\mathcal{W}_u^\nu(x,y)
=(xy)^{\nu+1/2} W_u^\nu(x,y), \quad
x,y\in (0,1),    
\end{equation*}
by using \eqref{eq:2.1} and 
\eqref{eq:1.2} we get 
$$
|\partial_y\mathcal{W}_t^\nu(x,y)|
\lesssim \Big(\frac{\sqrt{t}}{y}+1 \Big) \frac{
e^{-c|x-y|^2/t}
}{t}, \quad x,y\in (0,1) \text{ and } t>0.
$$
It follows that
\begin{align*}
|\mathcal{W}_u^\nu(x,y)-&\mathcal{W}_u^\nu(x,x_I)|
\le
\Big|\int_{x_I}^y
| \partial_z\mathcal{W}_u^\nu(x,z)
|dz\Big| 
\lesssim 
\Big|\int_{x_I}^y \Big(\frac{\sqrt{u}}{z}+1 \Big) \frac{e^{-c|x-x_I|^2/u}}{u}dz\Big|\\
&\lesssim 
|{x_I}-y| \Big(\frac{\sqrt{u}}{x_I}+1 \Big) 
\frac{
e^{-c|x-x_I|^2/u}
}{u},\qquad x\in (3I_j)\setminus (2I), \: y\in I.
\end{align*}
Therefore, 
\begin{align}\label{eq:Prop5.1_3Ij-2I}
& \|\mathcal{V}_\rho(\{\mathcal{P}_t^\nu\}_{t>0})(a)\|_{L^1((3I_j)\setminus(2I),dx)} \nonumber \\
& \qquad \lesssim 
\int_{(3I_j)\setminus(2I)}\int_I|a(y)|\int_0^\infty |{x_I}-y| \Big(\frac{\sqrt{u}}{x_I}+1 \Big) \frac{e^{-c|x-x_I|^2/u}}{u^2}du dydx \nonumber \\
&\qquad \lesssim 
\int_{(3I_j)\setminus(2I)} r_I \Big(\frac{1}{|x-x_I|x_I}+\frac{1}{|x-x_I|^2} \Big) dx \nonumber \\
& \qquad \lesssim 
 \frac{m((3I_j)\setminus(2I))}{x_I}+\int_{(0,\infty)\setminus(2I)}\frac{r_I}{|x-x_I|^2}dx
\lesssim 1.
\end{align}

\quad

On the other hand, since $3 I_j=(0,\frac{3}{2}2^{j}]$, if $x\in  (0,1)\setminus (3 I_j)=(\frac{3}{2}2^{j},1)$ and $y\in I\subseteq (\frac{39}{80}2^j,\frac{81}{80}2^j],$ then \begin{equation*}
|x-y|\ge x-{\frac{81}{80}2^{j}}\ge x-{\frac{81}{120}x=\frac{39}{120}x}.    
\end{equation*}

According to \eqref{eq:2.1}, we get
\begin{align}\label{eq:Prop5.1_3Ijc}
\|\mathcal{V}_\rho(\{\mathcal{P}_t^\nu\}_{t>0})(a)\|_{L^1((0,1)\setminus (3I_j),dx)}
&\lesssim
\int_{(0,1)\setminus (3I_j)}\int_I|a(y)|\int_0^\infty \frac{(xy)^{\nu+1/2}}{u^{\nu+2} }{
e^{-c |x-y|^2/u}
}dudydx \nonumber \\
& \lesssim
\int_{(0,1)\setminus (3I_j)}\int_I|a(y)|(xy)^{\nu+1/2}\int_0^\infty \frac{
e^{-c x^2/u}
}{u^{\nu+2} }dudydx \nonumber \\
& \lesssim m(I)^{-1}
\int_{\frac{3}{2}2^{j}}^1\int_I\frac{(xy)^{\nu+1/2}}{x^{2\nu+2} }dydx  \nonumber \\
&\lesssim
2^{j(\nu+1/2)}\int_{\frac{3}{2}2^{j}}^\infty\frac{dx}{x^{\nu+3/2}} 
\lesssim 1.
\end{align}
From
\eqref{eq:Prop5.1_2I},
\eqref{eq:Prop5.1_3Ij-2I}
and
\eqref{eq:Prop5.1_3Ijc},
we conclude that
$$
\|\mathcal{V}_\rho(\{\mathcal{P}_t^\nu\}_{t>0}(a)\|_{L^1((0,1),dx)}\le C,
$$
where $C$ does not depend on the atom $a.$\\

By proceeding as in 
\cite[pp. 284-285]{DPRS} we can now prove that if $a$ is a $S_\nu-$atom  such that there is not any interval $I$ satisfying that $\supp a\subseteq I$, $\|a\|_\infty\le m(I)^{-1}$ and $I\subseteq I_j^*$, for some $j\in\mathbb{Z}\setminus\{0\},$ then
$$
\|\mathcal{V}_\rho(\{\mathcal{P}_t^\nu\}_{t>0}(a)\|_{L^1((0,1),dx)}\le C,
$$
where $C$ is independent of $a$.
\end{proof}

The proof of Theorem \ref{Th:1.4} can be finished now similarly to the proof of Theorem \ref{Th:1.2}.

\bibliographystyle{siam}
\bibliography{references}


\end{document}